\documentclass{amsart}
\usepackage[all]{onlyamsmath}
\usepackage[l2tabu,orthodox]{nag}
\usepackage{fixltx2e}
\usepackage{amssymb}
\usepackage{amsbsy}
\usepackage{amsmath}
\usepackage{amsthm}
\usepackage{amsfonts}
\usepackage[colorlinks=true, breaklinks=true, urlcolor= black, linkcolor= black, citecolor= black, bookmarksopen=true,linktocpage=true,plainpages=false,pdfpagelabels]{hyperref}
\usepackage{verbatim}
\usepackage{enumerate}
\theoremstyle{plain}
\newtheorem{theorem}{Theorem}[section]
\newtheorem*{theorem*}{Theorem}
\newtheorem*{theoremA}{Theorem A}
\newtheorem*{theoremB}{Theorem B}
\newtheorem*{theoremC}{Theorem C}
\newtheorem*{theoremD}{Theorem D}
\newtheorem{lemma}[theorem]{Lemma}
\newtheorem{proposition}[theorem]{Proposition}

\newtheorem{fact}[theorem]{Fact}
\newtheorem*{fact*}{Fact}
\newtheorem{corollary}[theorem]{Corollary}
\newtheorem{claim}{Claim}
\numberwithin{claim}{theorem}
\newtheorem*{claim*}{Claim}

\theoremstyle{definition}
\newtheorem{definition}[theorem]{Definition}

\newtheorem*{definition*}{Definition}

\newtheorem*{notation*}{Notation}

\theoremstyle{remark}
\newtheorem{remark}[theorem]{Remark}
\newtheorem*{remark*}{Remark}

\newtheorem*{examples*}{Examples}

\newtheorem*{question*}{Question}

\newtheorem*{note*}{Note}

\newcommand{\N}{\mathbb{N}}
\newcommand{\Q}{\mathbb{Q}}
\newcommand{\CC}{\mathbb{C}}
\newcommand{\R}{\mathbb{R}}

\DeclareMathAlphabet\mathbfcal{OMS}{cmsy}{b}{n}

\newcommand{\F}{\mathbb{F}}

\newcommand{\Gl}{\mathbf{GL}}
\newcommand{\Gln}{\Gl_{n}}

\newcommand{\alg}[1]{{\ol{#1}}}
\newcommand{\sep}[1]{\ol{#1}^{\mathrm{sep}}}
\newcommand{\perf}[1]{\ol{#1}^{\mathrm{perf}}}
\newcommand{\trdeg}{\mathrm{trdeg}}
\newcommand{\ppow}[2][\relax]{{#2}_{p^{#1}}}
\newcommand{\ppowi}[1]{\ppow[\infty]{#1}}

\newcommand{\ra}[1][]{\xrightarrow{#1}}
\newcommand{\Set}[2]{\{#1\mid #2\}}
\newcommand{\defn}[1]{\mathbf{#1}} % definable set
\newcommand{\Mod}[1]{\mathrm{#1}} % model
\newcommand{\dD}{\defn{D}}
\newcommand{\dX}{\defn{X}}
\newcommand{\dY}{\defn{Y}}
\newcommand{\dZ}{\defn{Z}}
\newcommand{\dV}{\defn{V}}
\newcommand{\dW}{\defn{W}}
\newcommand{\dT}{\defn{T}}
\newcommand{\dU}{\defn{U}}
\newcommand{\mM}{\Mod{M}}
\newcommand{\mL}{\Mod{L}}
\newcommand{\mK}{\Mod{K}}
\newcommand{\ol}[1]{\overline{#1}}
\newcommand{\xx}{\times}
\newcommand{\dd}{\partial}
\newcommand{\Emph}[1]{\emph{#1}}
\newcommand{\Def}[1]{\Emph{#1}}

\newcommand{\Sort}[1]{\mathrm{\mathbf{#1}}}
\newcommand{\rf}{{\mathbf k}}
\newcommand{\vg}{{\mathbf \Gamma}}
\newcommand{\VF}{\Sort{VF}}
\newcommand{\RV}{\Sort{RV}}
\renewcommand{\O}{\mathcal{O}}
\newcommand{\bO}{\mathbfcal{O}}
\newcommand{\m}{\mathfrak{m}}
\newcommand{\boldmax}{\boldsymbol{\mathfrak{m}}}

\providecommand{\G}{}
\renewcommand{\G}{\mathcal{G}}
\newcommand{\bG}{\mathbfcal{G}}
\newcommand{\Gim}{\mathcal{G}^{\mathrm{im}}}
\newcommand{\bGim}{\mathbfcal{G}^{\mathrm{im}}}

\renewcommand{\k}{\mathbf{k}}
\newcommand{\K}{\Sort{K}}

\newcommand{\ACF}{\mathrm{ACF}}

\newcommand{\ACVF}{\mathrm{ACVF}}
\newcommand{\SCVF}{\mathrm{SCVF}}
\newcommand{\SCF}{\mathrm{SCF}}
\newcommand{\SCH}{\mathrm{SCH}}
\newcommand{\SCVH}{\mathrm{SCVH}}
\newcommand{\VDF}{\mathrm{VDF}_{\mathcal{EC}}}

\newcommand{\res}{\operatorname{res}}

\newcommand{\rv}{\operatorname{rv} }

\newcommand{\val}{\operatorname{val}}
\newcommand{\Res}{\operatorname{Res}}
\renewcommand{\div}{\operatorname{div}}
\newcommand{\latt}{\mathbf{S}}
\newcommand{\tor}{\mathbf{T}}

\renewcommand{\a}{\overline{a}}
\renewcommand{\b}{\overline{b}}
\newcommand{\x}{\overline{x}}
\newcommand{\y}{\overline{y}}

\newcommand{\dcl}{\mathrm{dcl}}
\newcommand{\cl}{\mathrm{cl}}
\newcommand{\acl}{\mathrm{acl}}
\newcommand{\St}{\mathrm{St}}
\newcommand{\tp}{\mathrm{tp}}

\newcommand{\car}{\mathrm{char}}
\newcommand{\Aut}{\mathrm{Aut}}

\newcommand{\code}[1]{\ulcorner #1\urcorner}
\newcommand{\eq}[1]{{#1}^{\mathrm{eq}}}
\newcommand{\NIP}{\mathrm{NIP}}

\newcommand{\clvK}[1]{\cl_{\val}^{K}(#1)}
\newcommand{\clvL}[1]{\cl_{\val}^{L}(#1)}

\renewcommand{\L}{\mathcal{L}}
\renewcommand{\phi}{\varphi}
\newcommand{\LOAG}{\L_{\mathrm{oag}}}
\newcommand{\LDIV}{\L_{\div}}
\newcommand{\LDIVpe}{\L_{\div,p,e}}
\newcommand{\LGamma}{\L_{\Gamma}}
\newcommand{\LGammape}{\L_{\Gamma,p,e}}
\newcommand{\LkGamma}{\L_{\Gamma k}}
\newcommand{\LkGammape}{\L_{\Gamma k,p,e}}
\newcommand{\LRV}{\L_{\mathrm{RV}}}
\newcommand{\LRVpe}{\L_{\mathrm{RV},p,e}}
\newcommand{\LG}{\L_{\G}}
\newcommand{\LGpe}{\L_{\G,p,e}}
\newcommand{\Lrg}{\L_{\mathrm{ring}}}

\newcommand{\Zar}[1]{\overline{#1}^{Zar}}

\def\Ind#1#2{#1\setbox0=\hbox{$#1x$}\kern\wd0\hbox to 0pt{\hss$#1\mid$\hss}
\lower.9\ht0\hbox to 0pt{\hss$#1\smile$\hss}\kern\wd0}
\def\Notind#1#2{#1\setbox0=\hbox{$#1x$}\kern\wd0\hbox to 0pt{\mathchardef
\nn="3236\hss$#1\nn$\kern1.4\wd0\hss}\hbox to 
0pt{\hss$#1\mid$\hss}\lower.9\ht0
\hbox to 0pt{\hss$#1\smile$\hss}\kern\wd0}
\def\ind{\mathop{\mathpalette\Ind{}}}

\newcommand{\Lp}{\L_{\mathrm{P}}}
\newcommand{\FF}{\mathbf{F}}
\newcommand{\lin}[1]{\mathbf{ld}_{#1}}
\newcommand{\flin}[1]{\ell_{#1}}
\newcommand{\Tp}{\mathrm{T}_{\mathrm{P}}}
\newcommand{\ldis}[1][]{\ind^{\mathrm{ld}}_{#1}}
\newcommand{\restr}[2]{\left. #1\right|_{#2}}
\newcommand{\Frac}[1]{\mathrm{Frac}(#1)}
\newcommand{\tensor}[1][]{\otimes_{#1}}
\newcommand{\tL}{\widetilde{\L}}
\newcommand{\tT}{\widetilde{T}}
\newcommand{\hL}{{\widehat{\L}}}
\newcommand{\hT}{{\widehat{T}}}
\newcommand{\Lpp}{\Lp^{\star}}
\newcommand{\Tpp}{\Tp^{\star}}
\newcommand{\substr}[1][]{\leq_{#1}}
\newcommand{\card}[1]{|#1|}
\newcommand{\subsel}{\preccurlyeq}
\newcommand{\Lpv}{\L_{\mathrm{P},\val}}
\newcommand{\Tpv}{\mathrm{T}_{\mathrm{P},\val}}
\newcommand{\FFRV}{\FF^{\RV}}
\newcommand{\Lppv}{\L_{\mathrm{P},\val}^{\star}}
\newcommand{\Tppv}{\Tpv^{\star}}
\newcommand{\Balls}{\mathbf{B}}

\newcommand{\tN}{\widetilde{N}}

\newcommand{\hN}{\widehat{N}}
\newcommand{\hA}{\widehat{A}}

\title{Imaginaries in separably closed valued fields}

\author{Martin Hils}
\address{Westf\"{a}lische Wilhelms-Universit\"{a}t M\"{u}nster, Institut f\"{u}r Mathematische Logik und Grundlagenforschung, Einsteinstr. 62, D-48149 M\"{u}nster, Germany}
\email{hils@uni-muenster.de}
\thanks{The first author was partially supported by ANR through ValCoMo (ANR-13-BS01-0006) and by DFG through SFB 878.}

\author{Moshe Kamensky}
\address{Department of Mathematics, Ben-Gurion University, Be'er-Sheva, 
Israel}
\email{kamenskm@math.bgu.ac.il}
\thanks{The second author was partially supported by 
ISF through grant no.~1382/15, 
and by ERC through FP7/2007-2013 (ERC Grant agreement no.~291111). 
The latter grant also funded a visit of the first author to Jerusalem 
in spring 2015 during which part of this research was carried out.}

\author{Silvain Rideau}
\address{University of California, Berkeley, Mathematics Department, Evans 
Hall, Berkeley, CA, 94720-3840, United States of America}
\email{silvain.rideau@berkeley.edu}
\thanks{The third author was partially supported by ANR through ValCoMo (ANR-13-BS01-0006).}

\keywords{Model Theory, Separably Closed Valued Fields, Classification of Imaginaries, Stable Domination}
\subjclass[2010]{Primary: 03C60; Secondary: 12J20, 03C45, 03C98, 03C10}

\date{\today }

\begin{document}

\bibliographystyle{plain}

\begin{abstract}
We show that separably closed valued fields of finite imperfection degree (either with \(\lambda\)-functions or commuting Hasse derivations) eliminate imaginaries in the geometric language. We then use this classification of interpretable sets to study stably dominated types in those structures. We show that separably closed valued fields of finite imperfection degree are metastable and that the space of stably dominated types is strict pro-definable.
\end{abstract}

\maketitle

%%%

\section{Introduction}

Since the work of Ax-Kochen and Ershov on the model theory of Henselian valued fields of residue characteristic 0, yielding, e.g., the approximative solution of Artin's conjecture via the famous transfer principle between $\Q_p$ and $\F_p((t))$ for large $p$, valued fields have been a constant source for deep model-theoretic applications. One may mention here Denef's rationality results for certain Poincar\'e series, and the foundations of motivic integration. 

But it is only with the work of Haskell, Hrushovski and Macpherson \cite{HaHrMa06,HaHrMa08} on the model theory of algebraically closed valued fields (\(\ACVF\)) that the methods of geometric model theory have been made available for the study of valued fields. Then, in their groundbreaking work (\cite{HrLo16}, see also \cite{Duc13}), Hrushovski and Loeser used these newly available tools to give a new account of the topological tameness properties of the  Berkovich analytification $\dV^{an}$ of a quasi-projective algebraic variety \(\dV\). The purpose of this paper is to study separably closed valued fields of finite imperfection degree from the point of view of geometric model theory, in the light of these results. 
%One of the motivations of our work is the potential applications to the study of Berkovich spaces. Indeed, the theory of separably closed valued fields seems well suited as a model-theoretic approach for certain finer questions about Berkovich spaces in positive characteristic, in particular in connection with separability issues. 

\medskip

Our first goal is the classification of the interpretable sets, also called 
\emph{imaginary sorts}. Recall that a set is interpretable if it is the 
quotient of a definable set by a definable equivalence relation. A theory is 
said to eliminate imaginaries if for every interpretable set $\dX$ there is a 
definable bijection between $\dX$ and a definable set, in other words, if the 
category of definable sets is closed under quotients. In valued fields, 
contrarily to what happens, e.g., in various contexts of algebraically closed 
fields with operators, it is not the case, in general, that the interpretable 
sets can all be understood in terms of the definable subsets of Cartesian 
powers of the field. There are some truly new quotients that appear and, in 
the case of \(\ACVF\), these new quotients are all described by Haskell, 
Hrushovski and Macpherson \cite{HaHrMa06} as certain quotients of linear 
algebraic groups by definable subgroups. They show that it is enough to add 
to the valued field sort, for any $n\geq 1$, the set of $\O$-lattices in 
$K^n$, which is given by $\mathbf{GL}_n(K)/\mathbf{GL}_n(\O)$, as well as the 
union of all $s/\mathfrak{m} s$, where $s$ is an $\O$-lattice in $K^n$. Here \(\O\) denotes the valuation ring of \(K\) and $\mathfrak{m}$ denotes the maximal ideal of $\O$. These sorts are called the 
\emph{geometric sorts} and we will denote them by $\bG$.

Since then, it has been shown that various theories of valued fields eliminate imaginaries down to the geometric sorts, namely the theory RCVF of real closed valued fields (by work of Mellor \cite{Mel06}), the theories of $p$-adic fields and of ultraproducts of $p$-adic fields (by work of Hrushovski, Martin and the third author \cite{HrMaRi14}), and the theory $\VDF$ of existentially closed valued differential fields $(K,v,\partial)$ of residue characteristic 0 satisfying $v(\partial(x))\geq v(x)$ for all $x$ (by work of the third author \cite{RidVDF}).

Let us now consider separably closed valued fields. Recall that if $K$ is a non perfect separably closed field of characteristic $p>0$, then $[K:K^p]=p^{e}$ for some $e\in\N^{*}\cup\{\infty\}$, and the elementary theory of $K$ is determined by $e$, the so-called \emph{imperfection degree} or \emph{Ershov invariant} of $K$. Hrushovski's model-theoretic proof of the relative Mordell-Lang Conjecture in positive characteristic \cite{Hru96} illustrates that the theory of separably closed non perfect 
fields -- a stable theory -- is a model-theoretic framework which provides a very useful approach for the study of questions from (arithmetic) algebraic geometry in positive characteristic.

In the valued context, the completions of the theory \(\SCVF\) of separably closed non-trivially valued fields are determined, as in the case without valuation, by the imperfection degree. One also has an explicit description of the definable sets, by work of Delon and, more recently, work of Hong \cite{Hon-QE}.

Our first result is that, as in the case of \(\ACVF\), the geometric sorts are sufficient to describe all the interpretable sets in separably closed valued field of finite degree of imperfection:

\begin{theoremA}[Theorem\,\ref{T:SCVF-EI}]
The theory $\SCVF_{p,e}$ of separably closed valued fields of finite degree of imperfection $e$, with the elements of a $p$-basis named by constants, eliminates imaginaries down to the geometric sorts.
\end{theoremA}

We prove this rather directly, reducing first to semi-algebraic sets using \(\lambda\)-functions and then performing a topological reduction to the corresponding result in $\ACVF_{p,p}$.  The crucial ingredient is Hong's density theorem from \cite{HonPhD} whose proof we include for convenience.

However, it seems more appropriate for practical purposes to work in the 
(strict) reduct obtained when, instead of working over a $p$-basis, one 
adds a sequence of $e$ commuting Hasse derivations (see \cite{ZieSCH}) to the 
language of valued fields. The situation is much trickier in this context. It 
is no longer the case that a definable subset of a Cartesian power of the 
field is necessarily in definable bijection with a semi-algebraic set. In 
order to reduce questions to $\ACVF_{p,p}$, prolongations come into the 
picture. In Corollary\,\ref{C:SCVH-EI}, we prove that the analogue of Theorem 
A also holds in \(\SCVH_{p,e}\), the theory of existentially closed valued 
fields with \(e\) commuting Hasse derivations.

Note that Theorem A follows formally from this second result, but we 
chose to present both proofs as the shorter topological proof seems 
interesting and instructive to us. The first proof consists in finding a 
canonical way of representing a semi-algebraic set definable in a separably 
closed valued field \(K\) by the \(K\)-points of a set definable in \(\alg{K}\models\ACVF_{p,p}\). The 
second proof, inspired by the work of the third author on the theory $\VDF$ 
(\cite{RidVDF}), is much more local. We only achieve a correspondence 
between \(K\) and \(\alg{K}\) at the level of types.

The main technical result which allows us, in the second proof, to reduce questions about definable sets to questions about types is the following density result for definable types, with parameters from the geometric sorts:

\begin{theoremB}[Theorem\,\ref{thm:SCVH dense def}]
Let $K\models\SCVH_{p,e}$, and let $\dX\subseteq\VF^{n}$ be a \(K\)-definable set. Let 
$A = \bG(\eq{\acl}(\code{\dX}))$. Then, there exists an $A$-definable type $p$ 
such that $p(x)\vdash x\in\dX$.
\end{theoremB}

Here, $\code{\dX}$ denotes the canonical parameter of $\dX$ in \(\eq{K}\).  
Recall that a type $p(x)$ over some structure $M$ is said to be definable if 
for every formula $\phi(x;y)$ there exists a formula $\theta(y)$ over \(M\), 
usually denoted by $\mathrm{d}_{p} x\,\phi(x;y)$, such that for all $m\in M$:
\[\phi(x;m)\in p\text{ if and only if }M\models\theta(m).\]

The proof of Theorem B follows the same strategy as in the case of \(\VDF\) mentioned above. The main new technical point is a quantifier elimination result for dense pairs of valued fields satisfying certain conditions. We then show that the pair $(\alg{K},K)$, for $K$ a model of \(\SCVH_{p,e}\), satisfies these conditions.

It follows immediately from this density statement that $\SCVH_{p,e}$, considered in the geometric sorts $\bG$, has weak elimination of imaginaries. That finite sets are coded in $\bG$ can easily be transferred from the corresponding result in $\ACVF_{p,p}$. This approach also has the added benefit of giving us, as a by-product, the fact that any type over an algebraically closed set has an automorphism invariant global extension, which is an important technical result for what follows.

In order to be able to classify imaginaries in the language of valued fields 
alone (in the case of finite imperfection degree, or even in the case of 
infinite degree of imperfection), it seems that one would need new ideas. In 
these cases, the goal would be to give a classification relative to those 
imaginaries which are definable in the field \emph{without valuation}.

\medskip

The second part of this paper is devoted to studying metastability and stably 
dominated types, first introduced by Haskell, Hrushovski and Macpherson 
\cite{HaHrMa06,HaHrMa08} to prove elimination of imaginaries in \(\ACVF\) 
down to the geometric sorts. A type is said to be stably dominated if its 
``generic'' extensions are controlled by pure stable interpretable subsets of 
the structure. A typical example is the generic type of the valuation ring 
$\bO$ which is controlled by the residue map. Haskell, Hrushovski and 
Macpherson show that in a model of \(\ACVF\), over the value group, there are 
``many'' stably dominated types. One says that \(\ACVF\) is \emph{metastable} 
and this gives a formal meaning to the idea that  a model of \(\ACVF\) is 
controlled in a very strong sense by its value group and its residue field.  
Precise definitions can be found in Section\,\ref{S:metastability}.

In \cite{HrLo16}, Hrushovski and Loeser use the machinery of geometric model theory in $\ACVF$ to construct a model-theoretic avatar $\widehat{\dV}$, whose points are given by the stably dominated types concentrating on $\dV$, of $\dV^{an}$, the Berkovich analytification of a quasi-projective algebraic variety \(\dV\). Among other things, they show that $\widehat{\dV}$ admits a definable strong deformation retraction onto a $\Sort{\Gamma}$-internal subset $\Sort{\Sigma}$, where $\Sort{\Gamma}$ is the value group. Since, the divisible ordered Abelian group \(\Sort{\Gamma}\) is the natural model-theoretic framework for piecewise linear geometry, this result implies that, without any smoothness assumption on \(\dV\), $\dV^{an}$ is locally contractible and admits a strong deformation retraction onto a piecewise linear space.

Our first result regarding these questions is that stable domination in a model $K$ of $\SCVH_{p,e}$ is characterized, via prolongations, by stable domination in the algebraic closure of $K$ (Proposition\,\ref{P:char st dom}). From this, we deduce the following result, using a description of definable closure obtained in 
Proposition~\ref{prop:descr dcl}:

\begin{theoremC}[Corollary \ref{C:SCVH-metastable}]
  The theories $\SCVH_{p,e}$ and $\SCVF_{p,e}$ are metastable over the value group $\vg$.
\end{theoremC}

Secondly, we establish the analogue of an important technical result from \cite{HrLo16}. Before we may state this result, we need to recall some notions. A \emph{pro-definable} set in $U$ is a set of the form $\dD=\varprojlim_{i\in I}\dD_i$, where $(\dD_i)_{i\in I}$ is a projective system in the category of definable sets and $I$ is a small index set. If all the $\dD_i$ and the transition maps are $C$-definable, $\dD$ is called $C$-pro-definable. A pro-definable function is a (bounded) family of definable functions (equivalently a function whose graph is pro-definable). The ($C$-)pro-definable sets form a category with respect to ($C$-)pro-definable maps. If the pro-definable set $\dD=\varprojlim_{i\in I}\dD_i$ is isomorphic to a pro-definable set with surjective transition functions, it is called \emph{strict pro-definable}. Equivalently, for every $i\in I$, the set $\pi_i(\dD)\subseteq \dD_i$ is definable (and not just type-definable). Here, $\pi_i$ denotes the projection map on the $i$th coordinate. Dually, one defines (strict) ind-definable sets. We refer to \cite[Section 
2.2]{HrLo16} for the basic properties of these notions.

Let $\dX$ be a $C$-definable set in \(\ACVF\). In \cite{HrLo16} it is shown that there is a strict $C$-pro-definable set $\widehat{\dX}$ such that for any $A\supseteq C$, the set $\widehat{\dX}(A)$ is in canonical bijection with the set of $A$-definable global stably dominated types $p(x)$ such that $p(x)\vdash x\in\dX$. If this is the case in a theory $T$, we say (rather loosely) that the set of stably dominated types in $T$ is strict pro-definable. From the proof in \cite{HrLo16}, we extract an abstract condition on a metastable NIP theory $T$ which implies that the set of stably dominated types is strict pro-definable in $T$. This yields the following:

\begin{theoremD}[Corollaries\,\ref{C:SCVF-strict-pro} and \ref{C:VDF-strict-pro}]
The set of stably dominated types is strict pro-definable in $\SCVH_{p,e}$ as well as in $\VDF$.
\end{theoremD}

The paper is organized as follows. In Section\,\ref{S:prelim} we gather some preliminaries about valued fields, the model theory of algebraically closed valued fields, separably closed fields as well as separably closed valued fields. We then present, in Section \ref{S:density}, the density theorem for semi-algebraic sets, and we prove Theorem A, namely that \(\SCVF_{p,e}\) eliminates imaginaries down to the geometric sorts. Section \ref{S:pairs} is devoted to the proof of Theorem B and of the fact that the geometric sorts classify the imaginaries even when working with Hasse derivations. This section mostly consists in the proof of a quantifier elimination result, of independent interest, for dense pairs of valued fields.

Subsequently, we show that various notions in $\SCVH_{p,e}$ reduce in the nicest possible way to $\ACVF_{p,p}$. In the short Section \ref{S:dcl-acl}, we establish this for the definable and the algebraic closure; Section \ref{S:metastability} gives 
a complete description of the stable stably embedded sets in $\SCVH_{p,e}$ (they are more or less the same as in $\ACVF_{p,p}$) as well as of the stably dominated types, in terms of stable domination of the prolongation in $\ACVF_{p,p}$. Putting all this together, we obtain the metastability of $\SCVH_{p,e}$ (Theorem C). In a final section we present an abstract framework, for metastable NIP theories, which guarantees the strict pro-definability of the set of stably dominated types, and we show that both $\SCVH_{p,e}$ and $\VDF$ fall under this framework, thus establishing Theorem D. 

\subsection{Acknowledgments}
This collaboration began during the Spring 2014 MSRI program {\em Model Theory, Arithmetic Geometry and Number Theory}. The authors would like to thank MSRI for its hospitality and stimulating research environment.

We are grateful to the referee for a thorough reading
of our paper, and for many useful suggestions.
%%% Local Variables:
%%% mode: latex
%%% TeX-master: "SCVF-EI"
%%% End:

\section{Preliminaries}\label{S:prelim}

Let us fix some notation. We will normally denote (ind-,pro-) definable sets 
in a given theory by bold letters (such as \(\dX\)). As customary, we will 
often identify such objects with their set of realizations in a universal 
domain (a fixed sufficiently saturated model), which we keep unspecified.  In 
such cases, by a \emph{set} (of parameters) we will mean a small subset of 
this universal domain.

Whenever $\dX$ is a definable set (or a union of definable sets) and $A$ is a 
set of parameters, $\dX(A)$ will denote $\dX\cap A$.  Usually in this 
notation there is an implicit definable closure, but we want to avoid that 
here because more often than not there will be multiple languages around and 
hence multiple definable closures which could be implicit. Similarly, if 
$\mathbfcal{S}$ is a set of definable sets, we will write $\mathbfcal{S}(A)$ 
for $\bigcup_{\Sort{S}\in\mathbfcal{S}}\Sort{S}(A)$.  When the language is 
clear, we will write \(A\substr{}K\) when \(A\) is a substructure of \(K\) 
(i.e., closed under function symbols).

We will be working with a fixed prime \(p\), and will write \(\ppow[n]{K}\) 
for the set \(\Set{x^{p^n}}{x\in{}K}\) of \(p^n\)-powers in a field \(K\).  
Likewise, if \(\dX\) is a definable field, \(\ppow[n]{\dX}\) is the definable 
set of \(p^n\) powers. We write \(\ppowi{K}\) 
(resp.~\(\ppowi{\dX}\)) for the intersection of \(\ppow[n]{K}\) 
(\(\ppow[n]{\dX}\)) over all \(n\).

\subsection{Imaginaries}

Recall that, in model theory, an imaginary point is a class of a definable 
equivalence relation. To every theory $T$, we can associate a theory $\eq{T}$ 
obtained by adding all the imaginary points. Every model \(M\) of \(T\) 
expands uniquely to a model \(\eq{M}\) of $\eq{T}$.
We write $\eq{\dcl}$ and $\eq{\acl}$ to denote the definable and algebraic 
closure in $\eq{M}$.

Let \(\dX\) be a set definable with parameters in a sufficiently saturated 
and homogeneous structure $M$. We denote by $\code{\dX}\subseteq \eq{M}$ the 
set of points fixed by the group \(G_{\code{\dX}}\) of automorphisms 
stabilizing $\dX$ globally. We say that  \(\dX\) has a \emph{canonical 
parameter} (or \emph{is coded}) if it can be defined over $\code{\dX}\cap M$.  
Likewise, \(\dX\) has an \emph{almost canonical parameter} (or is 
\emph{weakly coded}) if it can be defined over \(\eq{\acl}(\code{\dX})\cap 
M\), i.e., over some tuple with a finite orbit under the action of the same 
group. Note that if this finite orbit (viewed as a definable set) itself has 
a canonical parameter, then it is a canonical parameter for \(\dX\) as well.

A theory admits elimination of imaginaries precisely if every definable set 
has a canonical parameter.

\subsection{Valued fields}
\subsubsection{Notation and conventions}
Whenever $(K,\val)$ is a valued field, we will denote by 
$\vg(K):=\val(K^\star)$ its value group, by $\bO(K)$ its valuation ring, by 
$\boldmax(K)$ its maximal ideal, by $\rf(K) := \bO(K)/\boldmax(K)$ its residue ring and by 
$\res_K:\bO(K)\rightarrow \rf(K)$ the canonical projection. When the field $K$ 
is clear from context, we will write $\Gamma$, $\O$, $\m$, $k$ and $\res$. We 
usually identify  \(\Gamma\) and  \(\Gamma_\infty := \val(K) = 
\Gamma\cup\{\infty\}\). We will also denote by \(\VF(K)\) the points of the 
valued field itself (this will make more sense once we consider multi-sorted 
structures).

Let us now briefly recall the \emph{geometric sorts} from \cite{HaHrMa06}.  
For \(n\geq 1\), let \(\latt_n(K)\) be the set of \emph{\(\O\)-lattices} in 
\(K^n\), i.e.,  \(\latt_n(K)\simeq\Gln(K)/\Gln(\O)\). Note that, for 
any \(s\in \latt_n(K)\), the quotient \(s/\m s\) is an \(n\)-dimensional 
\(k\)-vector space. One puts \(\tor_n(K):=\dot{\bigcup}_{s\in \latt_n(K)}s/\m 
s\). Note that \(\tor_n(K)\) can similarly be identified with 
\((\Gln/\Sort{G})(K)\), where \(\Sort{G}\) is the inverse image of the stabilizer of a 
non-zero vector under the residue map on \(\Gln(\bO)\).  Hence it is indeed an 
imaginary sort.  The map associating to an element of \(\tor_n\) the 
corresponding lattice in \(\latt_n\) is denoted by  \(\tau_n\).

When \(n=1\), the map \(\val:\Gl_1(K)\ra\Gamma\) is a surjective homomorphism 
with kernel \(\Gl_1(\O)\), and therefore \(\latt_1\cong\vg\) canonically.  
Similarly, \(\rf^\star=(\bO/\boldmax)^\star\subseteq \tor_1\) canonically. In fact 
$\tor_1(K)$ is canonically isomorphic, as above, to the quotient of 
multiplicative groups $K^\star/(1+\m)$ which is often denoted $\RV(K)$.

\subsubsection{The valuation topology}
\begin{comment}
By a \Emph{definable topology} on a definable set \(\dX\) we mean an 
ind-definable set \(\dT\) of subsets of \(\dX\) that satisfies the definable 
analogue of the axioms of a topology. More precisely, it is an ind-definable 
set \(\dT\), an ind-definable subset \(\dU\subseteq\dX\xx\dT\), such that 
(setting \(\dU_a=\Set{x\in\dX}{(x,a)\in\dU}\)): each \(\dU_a\) is definable 
(rather than ind-definable), for all \(a,b\in\dT\), \(\dU_a\cap\dU_b=\dU_c\) 
for some \(c\in\dT\), and for any definable \(\dT_0\subseteq\dT\) (possibly 
with parameters), \(\exists{}t\in\dT_0(x\in\dU_t)=\dU_a\) for some 
\(a\in\dT\). As usual, a definable subset of the form \(\dU_a\) will be 
called open. We note that if \(\mM\) is a model, \(\dT(\mM)\) is not 
necessarily a topology on \(\dX(\mM)\), but it is the collection of definable 
open subsets in the topology generated by \(\dT(\mM)\). Other notions, such 
as a definable basis for \(\dT\), or continuous definable maps, are defined 
similarly.
\end{comment}

Let \(\tT\) be a theory of valued fields, with valued field sort \(\VF\).  We 
consider \(\VF^n\) with the definable \Def{valuation topology}, where a 
definable subset \(\dX\) is open if each of its points belongs to a product 
of open balls contained in \(\dX\). Here, \(\dX\), the balls and the points 
are over parameters, but if \(\dX\) is open and over \(K\), and 
\(a\in\dX(K)\) for some valued field \(K\), then the ball can also be chosen 
over \(K\). Continuous definable functions and other topological notions are 
defined similarly.

Given a valued field \(K\), the collection \(\dX(K)\), where \(\dX\) ranges 
over \(K\)-definable open subsets of \(\VF^n\), forms a basis for the usual 
valuative topology on \(K^n\). Note that an inclusion of valued fields 
$K\subseteq L$ is continuous if and only if $\vg(K)$ is a cofinal subset of 
$\vg(L)$.

\begin{lemma}\label{lem:dense geom}
  Assume that \(K\) is dense in \(L\) in the valuation topology. Then we have 
  \(\latt_n(K)=\latt_n(L)\) and \(\tor_n(K)=\tor_n(L)\) for every \(n\geq1\).  
  In particular, the extension \(L/K\) is immediate.
\end{lemma}

\begin{proof}
  Since \(K\) is dense in \(L\), the set \(K^N\) is dense in 
  \(L^N\) for every \(N\), and so \(\Gln(K)\) is dense in 
  \(\Gln(L)\), as \(\Gln(L)\) is an open subset of \(L^{n^2}\). For 
  any \(M\in\Gln(L)\), the set \(M\cdot\Gln(\bO(L))\) is open in \(\Gln(L)\), 
  so it contains some \(M_0\in\Gln(K)\), showing that 
  \(\latt_n(K)=\Gln(K)/\Gln(\bO(K))=\Gln(L)/\Gln(\bO(L))=\latt_n(L)\).

  Now let \(t\in \tor_n(L)\). By the previous paragraph, we know that 
  \(\tau_n(t)=s\in \latt_n(K)\), so applying a \(K\)-linear change of 
  variables we may assume that \(s=\bO^n\) and \(t\in \rf(L)^n=\bO(L)^n/\boldmax 
  \bO(L)^n\). But then \(\pi^{-1}(t)\) is an open subset of \(\bO(L)^n\), so by 
  the density assumption there is a tuple \(\a\in\bO(K)^n\) such that 
  \(\pi(\a)=t\). It follows that \(t\in{}\tor_n(K)\).
\end{proof}

Note that there are immediate extensions \(L/K\) such that \(K\) is not dense 
in \(L\), e.g., the Puiseux series field \(K=\bigcup_{n\in\N}\CC((t^{1/n}))\) 
inside the Hahn series field \(L=\CC((t^{\Q}))\).

Our goal now is to show that every smooth subvariety of affine space (viewed 
as a definable subset of \(\VF^n\)) is a topological manifold, i.e., 
definably locally homeomorphic to an open subset of \(\VF^m\) for some $m$. If \(\dX\) is 
a variety over a field \(K_0\), and \(a\in\dX(K_0)\), by a \Def{local 
coordinate system} around \(a\) we mean an \'etale map from a Zariski 
neighborhood of \(a\) to \(\VF^d\), all over \(K_0\), taking \(a\) to \(0\) 
(here \(d\) is the dimension of \(\dX\) at \(a\)).  The following observation 
is well known, see, e.g.,~\cite[Proposition~4.9]{MilneLEC} 
or~\cite[Proposition~I.3.24]{MilneEt}.
\begin{fact}\label{F:coords}
  If \(a\) is a smooth point of a variety \(X\), it admits a system of local 
  coordinates.
\end{fact}

The following statement is essentially the implicit function theorem in the 
valuative setting.

\begin{proposition}\label{P:manifold}
  Let \(\dX\) be a smooth subvariety of affine space, viewed as a definable 
  subset of \(\VF^n\) in a theory \(\tT\) of Henselian valued fields. Then 
  the induced definable topology on \(\dX\) is the unique topology for which 
  every local coordinate system around every point \(a\) of \(\dX\) is a 
  homeomorphism in a neighborhood of \(a\).
  
  In particular, \(\dX\) admits a unique definable topology making \(\dX(K)\) 
  locally homeomorphic to an open ball in \(K^d\) around each point.
\end{proposition}
\begin{proof}
  Uniqueness is clear by Fact~\ref{F:coords}. To show that every local 
  coordinate system is a local homeomorphism, we first note that the problem 
  is local for the Zariski topology on \(\dX\), and 
  by~\cite[Prop.~I.3.24]{MilneEt} we may therefore assume that \(\dX\) is the 
  zero set of a polynomial map 
  \(P(y_1,\dots,y_n)=(f_{d+1}(\y),\dots,f_n(\y)):\VF^n\ra\VF^{n-d}\), where 
  the tangent space of \(\dX\) at \(a\) is given as the kernel of \(dP(a)\), 
  and has dimension \(d\).

  Also, we are given a coordinate system 
  \(\bar{F}=(\bar{f_1},\dots,\bar{f_d}):\dX\ra\VF^d\), which we may assume to 
  be globally \'etale. Let \(F=(f_1,\dots,f_d)\) be any lift of \(\bar{F}\) 
  to a polynomial function on \(\VF^n\). Then \(dF\) restricts to a bijection 
  from the tangent space of \(\dX\) at \(a\) to \(\VF^d\). It follows that 
  the Jacobian matrix
  \((\dd{}f_i/\dd{}y_j)\) of the combined map \((F,P)\) is invertible.
  
  By rescaling, we may assume that both \(a\) and the coefficient of \(F,P\) 
  lie in \(\bO\). Replacing \(\dX\) by the graph of \(F\), we are in the 
  situation of Hensel's Lemma (as in, e.g., \cite[Thm.~9.14]{fvk} 
  or~\cite[Thm.~7.4]{PrZi}).  Thus, there is a value \(r\in\vg(K_0)\) 
  (namely, the valuation of the Jacobian determinant at \(a\)), such that for 
  any \(\x\in\VF^d\) with \(\val(\x)>2r\), there is a unique \(\y\in\dX\) 
  with \(F(\y)=\x\) and \(\val(\y-a)\ge\val(\x)-r\). This inequality also 
  gives the continuity.
\end{proof}

Note that this shows, in particular, that there is a valuation topology on 
any smooth affine variety, independent of an embedding into affine space, and 
therefore that the same definition determines a topology on any (not 
necessarily affine) smooth variety.

\subsection{Model theory of algebraically closed valued fields}
Recall the following languages for valued fields:
\begin{itemize}
  \item \(\LDIV=\Lrg\cup\{\div\}\) is the language with one sort \(\VF\), 
    where \(x\,\div\, y\Leftrightarrow\val(x)\leq\val(y)\).
  \item \(\LGamma\) is the two-sorted language with sorts \(\VF\) and 
    \(\vg\), given by \(\Lrg\) on \(\VF\), \(\LOAG=\{0,+,-,<,\infty\}\) on 
    \(\vg\) and \(\val:\VF\rightarrow\vg\).
  \item \(\LkGamma\) is the three-sorted language with sorts \(\VF\), \(\rf\) 
    and \(\vg\), given by \(\Lrg\) on \(\VF\), (another copy of) \(\Lrg\) on 
    \(\rf\), \(\LOAG\) on \(\vg\) and the functions 
    \(\val:\VF\rightarrow\vg\) and \(\Res:\VF^2\rightarrow \rf\) between the 
    sorts. Here,
    \[
    \Res(x,y):=\begin{cases}
      \res(\frac{x}{y}), \text{ if } \infty\neq\val(y)\leq\val(x);\\
      0, \text{otherwise.}
    \end{cases}
    \]
  \item \(\LRV\) is the $\RV$-language (also called language with leading 
    terms) with two sorts $\VF$ and $\RV$ where that last sort is interpreted 
    as $\VF^\star/(1+\boldmax)$. We add an element $0$ to $\RV$ for the image of 
    $0$. The language consists of the ring language on $\VF$, a map $\rv : 
    \VF\to\RV$ and the language $\LDIV$ on $\RV$. The symbol $\cdot$ 
    interprets the (multiplicative) group structure on $\RV$ and the symbol 
    $+$ the trace of the addition when it is well defined. To be precise, if 
    $\val(x) < \val(y)$, then $\rv(x) + \rv(y) = \rv(y) + \rv(x) = \rv(x)$, 
    if $\val(x) = \val(y) = \val(x+y)$, then $\rv(x) + \rv(y) = \rv(x + y)$ 
    and otherwise $\rv(x) + \rv(y) = 0$. Finally, $\rv(x)\div\rv(y)$ is 
    interpreted as $\val(x)\leq\val(y)$.

    Note that we will write $\sum_i x_i$ where $x_i \in\RV$. This notation is 
    slightly abusive as $+$ as defined above is not associative. What we mean 
    by $\sum_i x_i$ is in fact $\sum_{i\in I_0} x_i$ where $I_0$ is the set 
    of indices $i$ such that $x_i$ is of minimal valuation.
  \item \(\LG\) is the language in the geometric sorts (or geometric 
    language) from \cite[Section 3.1.]{HaHrMa06}, with set of sorts \(\bG := 
    \{\VF,\rf,\vg\}\cup
    \Set{\latt_n}{n\geq1}\cup\Set{\tor_n}{n\geq1}\). It is an extension of 
    \(\LkGamma\).  We use the notation from \cite{HaHrMa06}, although we 
    write the value group additively and not multiplicatively.
\end{itemize}

\begin{fact}\label{F:ACVF-QE}
  The theory \(\ACVF\) of algebraically closed non-trivially valued fields 
  eliminates quantifiers in either of the languages \(\LDIV\), \(\LGamma\), 
  \(\LkGamma\), \(\LRV\) and \(\LG\).
\end{fact}

For the first three languages, we refer to \cite[Theorem 2.1.1]{HaHrMa06}.  
(In the case of \(\LDIV\), the result is more or less due to Robinson.) The 
quantifier elimination result in the language \(\LG\) is \cite[Theorem 
3.1.2]{HaHrMa06}.  Quantifier elimination in \(\LRV\) for \(\ACVF\) is 
folklore and we are not aware that any proof exists in the literature. Let 
us give a sketch of the proof.
\begin{proof}[{Proof of Fact\,\ref{F:ACVF-QE}, \(\LRV\) case}]
  We have to show that given two models $M$ and $N$ of $\ACVF$ such that 
  $N$ is $|M|^+$-saturated and given any isomorphism \(f : A \to B\) between 
  $\LRV$-substructures of $M$ and $N$, respectively, we can extend $f$ to 
  $M$.  First, one can check that $f$ can be extended to the closure of $A$ 
  under inverses (both in $\VF$ and $\RV$). 

  Let $a\in\RV(A)$ be such that there is no $c\in\VF(A)$ with $\val(c) = 
  \val_{\RV}(a)$ where $\val_{\RV}$ is induced by $\val$ on $\RV$. If 
  $\val_{\RV}(a) \in\Q\tensor\val(A)$, let $n$ be minimal positive such that 
  $n\val_{\RV}(a) = \val(e)$ for some $e\in\VF(A)$. There exists $c\in\VF(M)$ 
  such that $c^n = e$ and $\rv(c) = a$. To show that this holds, it suffices 
  to prove it for $e = 1$ and that can be done easily by applying Hensel's 
  lemma and the Frobenius on the residue field (if the residue characteristic 
  is positive). Similarly there exists $d\in\VF(N)$ such that $d^n = f(e)$ 
  and $\rv(d) = f(\rv(a))$. If $\val_{\RV}(a) \not\in\Q\tensor\val(A)$, take 
  any $c$ such that $\rv(c) = a$ and any $d\in\VF(N)$ such that $\rv(d) = 
  f(\rv(a))$. Then one can extend $f$ by sending $c$ to $d$.

  Repeating this last step, we may assume that $\val(\VF(A)) = 
  \val_{\RV}(\RV(A))$ (and that $\RV(A)$ and $\VF(A)$ are closed under 
  inverses). Given \(r\in\RV(A)\), let \(a\in\VF(A)\) be an element with 
  \(\val(a)=\val_\RV(r)\). Then \(\frac{r}{\rv(a)}\) is a well defined 
  element \(c\) of \(\rf(A)\), so \(r=c\rv(a)\), and \(f(r)\) is uniquely 
  determined. Hence such an $f$ is completely determined by its reduct to 
  \(\LkGamma\) (actually the sort $\vg$ is useless here, but the two sorted 
  language with $\VF$ and $\k$ is not usually considered) and so $f$ extends 
  to $M$ by quantifier elimination in $\LkGamma$.
\end{proof}

We note that in the multi-sorted languages \(\LGamma\), \(\LkGamma\), \(\LRV\) 
and \(\LG\), all added sorts are interpretable in \(\LDIV\), and the 
structure is just the one induced by the corresponding interpretations in 
\(\LDIV\).

The completions of \(\ACVF\) are given by specifying the pair of 
characteristics \((\car(\VF),\car(\k))\in\Set{(0,0),(0,p),(p,p)}{p\,\,\text{a 
prime number}}\). The completion corresponding to \((p,q)\) is denoted by 
\(\ACVF_{p,q}\).

\medskip

The following is the main result of \cite{HaHrMa06}.

\begin{fact}[{\cite[Theorem 3.4.10]{HaHrMa06}}]\label{F:ACVF-EI}
  The theory \(\ACVF\) eliminates imaginaries in \(\LG\).
\end{fact}

\subsection{Separably closed fields}
\subsubsection{Notation and conventions}

Let $K$ be a field of characteristic $p > 0$. Let $b = (b_j)_{j\in J}$ be a 
(possibly infinite) tuple from $K$ and let $I : J \to p=\{0,\ldots,p-1\}$ be 
a function with finite support (that is a function that has value $0$ outside 
of a finite set). We denote by $b^I$ the monomial $\prod_{j\in J} 
b_j^{I(j)}$. The tuple $b$ is said to be a \emph{$p$-basis} of $K$ if the  
monomials $b^I$ form a linear basis of $K$ as a vector space over 
\(\ppow{K}\).  Then every $x\in K$ can be uniquely written as 
$x=\sum_{I}x_I^p b^{I}$. The $x_I$ are called the \emph{\(p\)-components of 
\(x\)} (with respect to $b$) and the functions $f_I : K\to K$ sending $x$ to 
$x_I$ are called the \emph{$p$-coordinate functions} or \emph{$\lambda$-functions}. Any characteristic $p$ 
field admits a $p$-basis and all $p$-bases of $K$ have the same cardinality 
$e$, usually called the \emph{imperfection degree} or the \emph{Ershov 
invariant} of $K$. Obviously, we have $[K:\ppow{K}] = p^e$ when \(e\) is 
finite.

A field $K$ is said to be \emph{separably closed} if it has no proper 
separable algebraic extension. For a prime $p$ and $e < \infty$, let 
$\L^{\lambda}_{p,e}:=\Lrg\cup\{b_1,\ldots,b_e\}\cup\Set{f_I}{I\in p^{e}}$ be 
the language with one sort \(\K\) and let  \(\SCF_{p,e,}\) be the theory of  
characteristic \(p\) separably closed fields with imperfection degree \(e\), 
where the $b_j$ form a $p$-basis with corresponding \(p\)-coordinate 
functions given by the \(f_I\). We will denote by 
\(\lambda^n:\K\to\K^{p^{ne}}\) the definable function whose coordinates are 
the $f_{I_n}\circ\cdots\circ f_{I_1}$ for all tuples $(I_n,\ldots,I_1)$. Note 
that $\lambda^n(x)^{p^n}$ is the tuple of \(\ppow[n]{\K}\)-coordinates of $x$ 
in the basis $b^{(I_n,\ldots,I_1)} = \prod_j (b^{I_j})^{p^j}$.

\begin{fact}
  The theory \(\SCF_{p,e,}\) eliminates quantifiers and imaginaries and is 
  complete. In case $e>0$, it is stable not superstable.
\end{fact}

The quantifier elimination result in that particular language is due to Delon 
\cite{DelSCF,DelSCF_MLBook}. The completeness result goes back to work of 
Ershov \cite{ErsSCF}. Stability and non superstability are proved in 
\cite{WooSCF}. Elimination of imaginaries is due to Delon and a proof can be 
found in \cite[Proposition\,3.9]{DelSCF_MLBook}.

The (type-definable) subfield \(\ppowi{\K}\) is the largest perfect 
subfield of \(\K\), and it is algebraically closed. The induced structure on 
\(\ppowi{\K}\) (by the ambient structure) is that of a pure 
algebraically closed field.

\subsection{Hasse derivations}

The $\lambda$-functions already give a ``field with operator'' flavor to 
separably closed fields. Separably closed fields can also be naturally 
equipped with more classical operators: Hasse derivations. As explained by 
Hoffmann \cite{HofSCH}, there are two natural ways in which to endow a 
separably closed field with Hasse derivations. Here, we follow Ziegler 
\cite{ZieSCH}.

A \Def{Hasse derivation} on a ring $R$ is a sequence $D = (D_n)_{n\in\N}$ of 
additive functions $D_n:R\to R$ such that for all $x$, $y\in R$, $D_0(x) = x$ 
and $D_n(x y) = \sum_{k+l = n} D_k(x)D_l(y)$. We say that $D$ is 
\Def{iterative} if $D_m\circ D_n = \binom{m+n}{n}D_{m+n}$ also holds. We will assume all Hasse derivations to be iterative.

Let $e\in\N$. Let $K$ be a field of characteristic $p > 0$ and 
$(D_1,\ldots,D_{e})$ be a tuple of commuting Hasse derivations on $K$ (i.e.,  
$D_{i,n}\circ D_{j,m} = D_{j,m}\circ D_{i,n}$ for all $i,j \leq e$ and 
$n,m\in\N$). It is easy to check that \(\ppowi{K}\) is contained in 
the field
\[
C_\infty:=\Set{x\in K}{D_{i,n}(x)=0\text{ for all \(i\leq e\) and \(n>0\)}}
\]
of  \emph{(absolute) constants}. 

The field $K$ is said to be \Def{strict} if
\[\ppow{K}=C_1:=\Set{x\in{}K}{D_{i,1}(x)=0\text{ for all \(i\leq{}e\)}}.\]
(We then have $\ppowi{K}=C_\infty$). Let
\[\L_{p,e}^{D}:=\Lrg\cup\Set{D_{i,n}}{0<i\leq e\text{ and }n\in\N}\]
and let  \(\SCH_{p,e,}\) be the theory of characteristic \(p\) separably 
closed strict fields of imperfection degree $e$ with $e$ commuting Hasse 
derivations $D_i = (D_{i,n})_{n\in\N}$.

For $N\in\N^e$, we will denote $D_{N}(x) = D_{1,n_1}\circ\ldots\circ 
D_{e,n_{e}}(x)$ and $D_\omega(x) = (D_{N}(x))_{N\in\N^e}$.

Note that any separably closed field of imperfection degree $e$ (and 
$p$-basis $b$) can be made into a strict field with $e$ commuting Hasse 
derivations by setting
\[D_{i,n}(b^I) = \binom{I(i)}{n}b_i^{I(i)n}\prod_{j\neq i}b_j^{I(j)}\]
and
\[D_{i,n}(x) = \sum_I\lambda^m_I(x)^{p^m}D_{i,n}(b^I)\]
for any $m$ such that $n < p^m$. For all $n>0$, we then have $D_{i,n}(b_j) = 
1$ if $i=j$ and $n=1$ and $D_{i,n}(b_j) = 0$ otherwise. Such a $p$-basis is 
said to be \Def{canonical}. Conversely, if $b$ is a canonical $p$-basis, the 
$D_{i,n}$ can be expressed as above using $b$ and $\lambda$.

\begin{fact}[\cite{ZieSCH}]
  The theory $\SCH_{p,e}$ eliminates quantifiers and imaginaries and is 
  complete.
\end{fact}

As noted in \cite{ZieSCH}, the quantifier elimination result can be deduced 
from quantifier elimination in $\SCF_{p,e}$. This remains true in the valued 
setting, as will be seen below.

\subsection{Separably closed valued fields}

We now consider a separably closed field $K$ of positive characteristic $p$ 
and finite imperfection degree $e$ endowed with a non-trivial valuation 
$\val$. As is the case for algebraically closed valued fields, there is a 
number of natural languages in which to consider these structures:
\begin{itemize}
  \item The one sorted language 
    $\LDIVpe^{\lambda}:=\LDIV\cup\L_{p,e}^{\lambda}$;
  \item The two sorted language 
    $\LGammape^{\lambda}:=\LGamma\cup\L_{p,e}^{\lambda}$;
  \item The three sorted language 
    $\LkGammape^{\lambda}:=\LkGamma\cup\L_{p,e}^{\lambda}$;
  \item The leading term language $\LRVpe^{\lambda} := 
    \LRV\cup\L_{p,e}^{\lambda}$
  \item The geometric language $\LGpe^{\lambda}:=\LG\cup\L_{p,e}^{\lambda}$.
\end{itemize}

Let $\SCVF_{p,e}$ denote the theory of separably closed non-trivially valued fields of 
characteristic $p$ and imperfection degree $e$ (with $\lambda$-functions) in 
either of these languages. Similarly we define $\LDIVpe^{D}$, 
$\LGammape^{D}$, $\LkGammape^{D}$, $\LRVpe^{D}$ and $\LGpe^{D}$ to be the 
languages with $e$ Hasse derivations and we denote by $\SCVH_{p,e}$ the 
theory (in any of these languages) of separably closed strict non-trivially valued fields 
of imperfection degree $e$ with $e$ commuting Hasse derivations. We will also 
denote by $\SCVF$ the theory of separably closed non-trivially valued fields in the 
language $\LDIV$.

\begin{proposition}\label{prop:SCVF-dense}
  Let $K$ be a separably closed non-trivially valued field. For all $n\in\N$, $\ppow[n]{K}$ 
  is dense in $\alg{K}$. Moreover, if $K$ is $\omega$-saturated, then 
  $\ppowi{K}$ is dense in $\alg{K}$. In particular 
  $\latt_n(\ppowi{K}) = \latt_n(\alg{K})$ and 
  $\tor_n(\ppowi{K}) = \tor_n(\alg{K})$ in this case.
\end{proposition}

\begin{proof}
  Cf. \cite[Lemma\,5.2.5 and Remark\,5.2.6]{HonPhD}. The statement about 
  $\ppowi{K}$ follows by compactness and the one about the $\latt_n$ 
  and $\tor_n$ is a consequence of Lemma\,\ref{lem:dense geom}
\end{proof}

In \cite{HonPhD}, Hong proved that $\SCVF_{p,e}$ eliminates quantifiers in 
the language with two sorts. We now show that his proof generalizes to any of 
the five languages we have been considering. Let us also mention that Hong 
later proved, in \cite{Hon-QE}, a stronger quantifier elimination result 
using "parametrized $\lambda$-functions" which also covers the case of 
separably closed valued fields of infinite Ershov invariant.

\begin{proposition}\label{prop:SCVF EQ}
  $\SCVF_{p,e}$ eliminates quantifiers in the one, two and three sorted 
  languages, the leading term language as well as in the geometric language.
\end{proposition}

In the following pages, let $\L$ denote any of the five languages $\LDIV$, 
$\LGamma$, $\LkGamma$, $\LRV$ and $\LG$ and 
$\L^\lambda=\L\cup\L_{p,e}^\lambda$ the corresponding enrichment with 
$\lambda$-functions (for some fixed \(p\) and \(e\)).  The proof relies on 
one main technical tool, $\lambda$-resolutions:

\begin{lemma}[{\cite[Definition-Proposition 5.2.2]{HonPhD}}]\label{lem:lambda 
  res}
  Let $M\models\SCVF_{p,e}$, $A\leq M$ be a substructure and $\dX\subset 
  \VF^m$ be quantifier free $\L^\lambda(A)$-definable. There exists $n\in\N$ 
  such that $\lambda^n(\dX) = \dY$ for some quantifier free $\L(A)$-definable 
  set $\dY\subseteq \VF^{mp^{en}}$. Such a set $\dY$ is called a 
  \emph{$\lambda$-resolution }of $\dX$.
\end{lemma}

Note that once we know quantifier elimination, \emph{all} definable subsets 
of $\VF^{n}$ will have a $\lambda$-resolution.

\begin{proof}
  This is an immediate consequence of the fact that $\lambda(x+y)$ and 
  $\lambda(x y)$ can be written as polynomials in $\lambda(x)$ and 
  $\lambda(y)$, and that $\lambda^n$ is onto.
\end{proof}

\begin{lemma}\label{lem:ext transc}
  Let $M$ and $N$ be two non-trivially valued fields (considered as $\L$-structures), 
  $A\leq M$, $f : A \to N$ be an $\L$-embedding and $a\in\VF(M)$ be 
  transcendental over $\VF(A)$. Assume that $N$ is $|A|^+$-saturated and 
  $\VF(N)$ is dense in $\alg{\VF(N)}$. Then $f$ can be extended to an 
  $\L$-embedding $A(a) \to N$.
\end{lemma}

Before giving the proof, we mention that in \(\ACVF\), any definable function 
from an imaginary geometric sort to \(\VF\) has finite image. This follows 
for example from the fact that there are uncountable models of \(\ACVF\) 
whose imaginary part is countable (along with quantifier elimination for 
\(\ACVF\)). See Lemma~\ref{lem:acl field} for a similar argument for 
\(\SCVF\) or \(\SCVH\).

\begin{proof}
  By compactness, it suffices to prove that for every quantifier free 
  $\L(A)$-definable set $\dX$ such that $a\in \dX(M)$, $f_\star \dX\cap 
  \VF(N)\neq\emptyset$. Since \(N\) is non-trivially valued, \(\alg{\VF(N)}\) 
  is a model of \(\ACVF\), and since $\ACVF$ eliminates quantifiers in $\L$, 
  $f_\star \dX\cap \alg{\VF(N)}\neq\emptyset$. If $f_\star \dX\cap 
  \alg{\VF(N)}$ has non empty interior, then we conclude by density of 
  $\VF(N)$ in $\alg{\VF(N)}$. If it has empty interior, then $f_\star \dX\cap 
  \alg{\VF(N)}$ is finite and so is $\dX\cap\alg{\VF(M)}$.  In particular it 
  follows that $a$ is algebraic (in $\ACVF$) over $A$ and hence (by the 
  remark preceding the proof) over $\VF(A)$. This contradicts the fact that 
  $a$ is transcendental over $\VF(A)$.
\end{proof}

In fact, $f(a)$ can be chosen in any dense subfield $K_0$ of $\VF(N)$.

\begin{corollary}\label{cor:ext separable}
  Let $M$ and $N$ be two non-trivially valued fields (considered as $\L$-structures), 
  $A\leq M$, $f : A \to N$ be an $\L$-embedding. Assume that $N$ is separably 
  closed and $|M|^+$-saturated. Let $\VF(A)\leq K_0\leq \VF(M)$ be such that 
  the extension $\VF(A)\leq K_0$ is separable. Then $f$ can be extended to an 
  $\L$-embedding $A(K_0) \to N$.
\end{corollary}

\begin{proof}
  First, if $K_0$ is an algebraic extension of $\VF(A)$, then, by quantifier 
  elimination in $\ACVF$, $f$ extends to an $\L$-embedding of $A(K_0)$ into 
  $N$ (a priori, the image of $f$ is in $\alg{N}$, but because $N$ is 
  separably closed, it is, in fact, in $N$). So we may always assume that 
  $\VF(A)$ is separably closed.

  By Proposition \ref{prop:SCVF-dense}, $N$ is dense in $\alg{N}$. By 
  compactness and saturation of $N$, it is enough to show the result for 
  $K_0$ which is a finitely generated separable extension of $\VF(A)$. Note 
  that such an extension admits a separating transcendence basis. Using Lemma 
  \ref{lem:ext transc}, we may thus conclude by induction on 
  $\trdeg(K_0/\VF(A))$.
\end{proof} 

\begin{proof}[{Proof of Proposition\,\ref{prop:SCVF EQ} (following 
  {\cite[Proof of 5.2.1, p.59]{HonPhD}})}]
  Let $M$, $N$ be two models of
 $\SCVF_{p,e}$, $A\leq M$ and $f : A\to N$ an 
  $\L^\lambda$-embedding. We have to show, given $a\in M$ and provided $N$ is 
  $|M|^+$-saturated, that $f$ extends to $A(a)$. By compactness, it suffices 
  to show that for every quantifier free $\L^\lambda(A)$-definable set $\dX$ 
  such that $a\in \dX(M)$, $f_\star \dX \cap N\neq\emptyset$. By 
  Lemma\,\ref{lem:lambda res}, we may assume that $\dX$ is $\L(A)$-definable, 
  at the cost of turning $a$ into a finite tuple of elements of $\VF(M)$.  
  Note that, because $A$ is closed under $\lambda$-functions, that 
  $\VF(M)/\VF(A)$ is a separable field extension, so we may extend $f$ to an 
  $\L$-embedding $M\to N$ by Corollary \ref{cor:ext separable}. We then have 
  $f(a) \in f_\star \dX \cap N$ and that concludes the proof.
\end{proof}

\begin{corollary}\label{cor:SCVF compl}
  Let $M\models\SCVF_{p,e}$. The theory of $M$ is completely determined by 
  the $\LDIV$-isomorphism type of $\F_p[b_1,\ldots,b_e]$.
\end{corollary}

From quantifier elimination (Proposition \ref{prop:SCVF EQ}), the existence 
of $\lambda$-resolutions and the fact that $\ACVF_{p,p}$ is $\NIP$, we obtain 
the following.

\begin{corollary}[{Delon, see \cite[Corollary 
  5.2.13]{HonPhD}}]\label{C:SCVF-NIP}
  Any completion of $\SCVF_{p,e}$ is $\NIP$.
\end{corollary}

Similar results hold for $\SCVH_{p,e}$:

\begin{proposition}\label{prop:SCVH QE}
  $\SCVH_{p,e}$ is complete and eliminates quantifiers in the one, two and 
  three sorted languages, the leading term language as well as in the 
  geometric language.
\end{proposition}

The proof uses the following notion:

\begin{definition}
  Let $M\models\SCVH_{p,e}$ and $A\subseteq M$. A $p$-basis $b$ of $M$ is 
  said to be \emph{very canonical} over $A$ if it is a canonical $p$-basis 
  such that $b_i\in\bO$ for all $i$ and the $\res(b_i)$ are algebraically 
  independent over $\k(A)$.
\end{definition}

\begin{lemma}
  Let $M\models\SCVH_{p,e}$ and $A\subseteq M$. If $M$ is $|A|^+$-saturated, 
  then $M$ admits a $p$-basis which is very canonical over $A$.
\end{lemma}

\begin{proof}
  By \cite[Corollary\,4.2]{ZieSCH}, $M$ has a canonical $p$-basis. As 
  $D_{i,n}(\ppowi{\VF})=0$ for any $n>0$, if $b$ is a canonical 
  $p$-basis and $a$ is an $e$-tuple from $\ppowi{\VF}(M)$, then $b+a$ 
  is also a canonical $p$-basis. Since $\ppowi{\VF}(M)$ is dense in 
  $\VF(M)$, for any $c\in\VF(M)$, $\res(c+\ppowi{\VF}(M)) = \k(M)$.  
  Using $|A|^+$-saturation, one can find a tuple $c\in\k(M)$ algebraically 
  independent over $\k(A)$ and a tuple $a\in\ppowi{\VF}$ such that $\res(b+a) 
  = c$.
\end{proof}

\begin{proof}[Proof of Proposition\,\ref{prop:SCVH QE}]
  Let us first prove quantifier elimination in the leading term language.  
  Quantifier elimination in the one, two and three sorted languages follow 
  formally. To be exact, in the case of the three sorted language, a little 
  work is still necessary as there is no sort for $\vg$ in $\LRV$. But, when 
  doing a back and forth, that can be easily taken care of by lifting points 
  in $\vg$ to $\RV$ using quantifier elimination in $\ACVF$. We will be using 
  Ziegler's trick from \cite{ZieSCH}.

  Let $\phi(x)$ be an $\LRV^{D}$-formula. As noted above, the $D_{i,n}$ can 
  be expressed in terms of $b$ and $\lambda$. It follows that $\phi(x)$ is 
  equivalent to an $\LRV^{\lambda}$-formula $\psi(x)$ and that by 
  Proposition\,\ref{prop:SCVF EQ}, we may assume that $\psi$ is quantifier 
  free. Note that, by Corollary\,\ref{cor:SCVF compl}, the formula $\psi$ 
  does not depend on the actual choice of very canonical $p$-basis.

  The only occurrences of $b$ and $\lambda$ in $\psi$ are in terms of the 
  form $\rv(\sum_I P_I(x)b^I)$ where the $P_I$ and $Q_I$ are polynomial in 
  $\lambda^m(x)$ for some $m\in\N$  (provided we rewrite $\sum_I P_I(x)b^I = 
  0$ into $\rv(\sum_I P_I(x)b^I) = 0$). Applying the Frobenius automorphism 
  $m$-times, we may assume that the $P_I$ and $Q_I$ are polynomials in 
  $\lambda^m(x)^{p^m}$. But by \cite[Lemma\,4.3]{ZieSCH}, the 
  $\lambda^m(x)^{p^m}$ can be expressed as polynomials in the $D_{i,n}(x)$ 
  and $b$. We may therefore assume that the $P_I$ and $Q_I$ are polynomials 
  in the $D_{i,n}(x)$.

  Taking $b$ to be a very canonical basis over $\langle x \rangle$, the 
  structure generated by $x$ (which we can do because the choice of very 
  canonical $p$-basis did not matter so far), we get that \[\rv\left(\sum_I 
  P_I(x)b^I\right) = \sum_I\rv(P_I(x))\res(b)^I.\]

  Moreover, for all $t_I\in\rv(\langle x \rangle)$, $\sum_I t_I\res(b^I) = 0$ 
  if and only if $\bigwedge_It_I=0$. Therefore, $\psi$ can be rewritten so 
  that $b$ does not appear in $\psi$ (and the rewriting does not depend on 
  the point $x$ we are considering). This concludes the proof in the leading 
  term language.

\smallskip

  Let us now prove quantifier elimination in the geometric language. Let $M$ 
  and $N$ be models of $\SCVH_{p,e}$ in the geometric language, $A\leq M$ and 
  $f:A\to N$ an $\LG^D$-embedding. Assume that $M$ is $\omega$-saturated and 
  $N$ is $|M|^+$ saturated. 
  %By extracting $p$-th roots of constants (recall that Hasse derivations have a unique extension to algebraic 
  %extensions), we may assume that $\VF(A)$ is strict, see 
  %\cite[Lemma\,2.4]{ZieSCH}. 
  We may assume that $\VF(A)$ is strict. Indeed, by \cite[Lemma\,2.4]{ZieSCH}, there is a smallest strict 
  extension of $\VF(A)$ which is uniquely determined up to isomorphism -- as a field extension, it is algebraic and purely inseparable.
  
  It now follows from 
  \cite[Corollary\,2.2]{ZieSCH} that the field extension 
  $\VF(A)\subseteq\VF(M)$ is separable. By Corollary \ref{cor:ext separable}, 
  $f$ extends to an $\LG$-embedding of $M$ into $N$, so there exists in 
  particular an $\LG$-embedding $g: B := A(\ppowi{\VF}(M))\to N$ extending 
  $f$.  Clearly $g(\ppowi{\VF}(M))\subseteq\ppowi{\VF}(N)$.  Since we have 
  added only absolute constants, $g$ is automatically an $\LG^D$-embedding.  
  As $M$ is $\omega$-saturated, $\ppowi{\VF}(M)$ is dense in $\VF(M)$ by 
  Proposition \ref{prop:SCVF-dense}, and so \(B\) is generated by \(\VF(B)\).  
  We may now use quantifier elimination in the one sorted language to extend 
  $g$ to $M$.  This concludes the proof in the geometric language.
\end{proof}

We will sometimes consider the set of all geometric sorts except the valued 
field sort $\VF$. We will call these sorts the \emph{imaginary} geometric 
sorts and write $\bGim:=\bG\setminus\{\VF\}=\{\vg,\rf\}\cup\{\latt_n | 
n\geq1\}\cup\{\tor_n | n\geq1\}$.

\begin{corollary}\label{cor:purely-geometric}
  Let $K\models T$ and $A\substr K$, where $T=\SCVH^{\G}_{p,e}$ (or 
  $T=\SCVF^{\G}_{p,e}$, respectively). Let $L=\alg{K}$. Then 
  $\Gim(K)=\Gim(L)$, and restriction to \(K\)-points determines an 
  equivalence between the $\LG^D(A)$-definable sets (the 
  $\LG^\lambda(A)$-definable sets, respectively) in the multi-sorted 
  structure $\Gim(K)$ and the $\LG(A)$-definable sets in $\Gim(L)$.
\end{corollary}

\begin{proof}
  We may assume that $T=\SCVH^{\G}_{p,e}$, as the result for 
  $\SCVF^{\G}_{p,e}$ is a consequence of the one for $\SCVH^{\G}_{p,e}$. By 
  Proposition \ref{prop:SCVF-dense}, we have $\Gim(K)=\Gim(L)$. The statement 
  about definable sets follows from quantifier elimination for $\SCVH_{p,e}$ 
  in the language $\LG^D$ (Proposition \ref{prop:SCVH QE}).
\end{proof} 

\begin{remark}\mbox{}
  \begin{enumerate}
    \item
      Note that in $\SCVF_{p,e}$ or $\SCVH_{p,e}$, the field of absolute 
      constants is not stably embedded (although, since $\VF\subseteq 
      \dcl(\ppow[n]{\VF})$, each of the $\ppow[n]{\VF}$ is).  
      
      Indeed, let $M$ be an $\omega$-saturated model, 
      $a\in\VF(M)\setminus\ppowi{\VF}(M)$, and let $B$ be the set of 
      balls in $M$ that contain $a$. If $\ppowi{\VF}(M)$ were stably 
      embedded, then, by quantifier elimination, $B$ would be 
      \(\LG(\ppowi{\VF}(M))\)-definable and by definable spherical 
      completeness of \(\ACVF\), there would exist 
      \(c\in\bigcap_{b\in{}B}b(\ppowi{\VF}(M))\). But 
      \(\bigcap_{b\in{}B}b(M)=\{a\}\) and $a\not\in\ppowi{\VF}(M)$.

    \item
      Nevertheless, by quantifier elimination, $\ppowi{\VF}$ is a pure 
      algebraically closed valued field in the following (weak) sense: any 
      definable set in $\ppowi{\VF}(M)$ (including the geometric sorts) is the 
      intersection of a quantifier free $\LG(M)$-definable set with (some 
      Cartesian power of) $\ppowi{\VF}$.
  \end{enumerate}
\end{remark}

%%% Local Variables:
%%% mode: latex
%%% TeX-master: "SCVF-EI.tex"
%%% End:

\section{Imaginaries and density}\label{S:density}

\subsection{The density theorem}

\begin{lemma}\label{L:Subvar-SCF}
  Let \(\dV\) be an irreducible variety defined over \(K=\sep{K}\). Then any 
  \(\alg{K}\)-definable open subvariety \(\Sort{U}\subseteq \dV\) is defined over 
  \(K\).
\end{lemma}

\begin{proof}
  We may suppose that \(\mathrm{char}(K)=p>0\), and we may assume that 
  \(\dV\) is affine and \(\Sort{U}=\dV_f\), where 
  \(f\in{}\alg{K}[\dV]=K[\dV]\otimes_{K}\alg{K}\). So there is some \(N=p^k\) 
  such that \(f^N\in K[\dV]\). Thus, \(\dV_f=\dV_{f^N}\) is defined over 
  \(K\).
\end{proof}

The valuation topology on powers of a model \(\mL\) of \(\ACVF\) determines, 
in the terminology of~\cite[2.11]{vdd}, a \emph{topological system} (using 
the language \(\LDIV\)): All polynomials are continuous, and punctured balls 
are open. Further, it satisfies the assumptions of~\cite[2.15]{vdd}, and 
therefore we have:

\begin{fact}[{\cite[2.18]{vdd}}]\label{F:dim}
  In \(\ACVF\), a set \(\dX\subseteq\VF^n(L)\) definable with parameters is 
  Zariski dense if and only if it contains a non-empty open ball.
\end{fact}

Using Proposition\,\ref{P:manifold}, we obtain the same result for any smooth 
variety:

\begin{corollary}\label{C:dim}
  A definable subset $\dX$ of a (connected) smooth algebraic variety \(\dV\) 
  in \(\ACVF\) is Zariski dense if and only if it contains a non-empty 
  subset which is open in \(\dV\) for the valuation topology.
\end{corollary}
\begin{proof}
  The claim is local, hence we may choose local coordinates as in 
  Proposition\,\ref{P:manifold} and reduce to the case \(\dV=\VF^n\). Now the 
  claim follows from Fact~\ref{F:dim}.
\end{proof}

\begin{lemma}\label{L:strat}
If \(\mK\models\SCVF\), and \(\dX\subseteq\VF^n\) is a semi-algebraic subset defined over \(K\), then there are absolutely irreducible affine subvarieties \(\dY_i\) of \(\VF^n\) defined over \(K\) and \(\dX_i\subseteq \dY_i\) given by a Boolean combination of conditions of the form \(\val(h)>0\), where \(h\) is an invertible regular function on \(\dY_i\), such that \(\dX=\bigcup_i\dX_i\).
\end{lemma}

\begin{proof}
  Let \(J\) be the set of polynomials \(F\) over \(K\) such that for some 
  polynomial \(G\), \(F\div{}G\) or \(G\div{}F\) occurs in a semi-algebraic 
  definition of \(\dX\), and for every subset \(I\) of \(J\), let \(\dW_I\) 
  be the locally closed subset of \(\VF^n\)  given by
  \begin{equation*}
    \bigwedge_{F\in I}F(\x)=0\wedge\prod_{F\not\in I}F(\x)\neq0
  \end{equation*}
  For each \(I\), \(\dW_I\) is an affine subvariety of \(\VF^n\), and 
  \(\dX\cap{}\dW_I\) is given by a Boolean combination of polynomial 
  equations and valuative inequalities as in the definition.  Hence, by 
  restricting to \(\dW_I\), we may assume that \(\dX\) itself was given by 
  such a Boolean combination to begin with.

  Writing \(\dX\) in disjunctive normal form, we present \(\dX\) as a finite 
  union of definable sets \(\dX_i\), each the intersection of a (Zariski) 
  locally closed subset \(\dZ_i\) with valuative inequalities. Let \(\dY_i\) 
  be the Zariski closure of \(\dX_i(K)\) in \(\dZ_i\) (over \(K\)). The 
  irreducible components of \(\dY_i\) are defined over \(K\), since \(K\)-points are dense. Hence we may assume that each \(\dY_i\) is (absolutely) 
  irreducible.  Then they satisfy the requirements of the claim.
\end{proof}

\begin{proposition}[Hong {\cite[Theorem 5.3.1]{HonPhD}}]\label{P:vdense}
  Let \(\mK\models\SCVF\) and \(\dX\subseteq \VF^n(K)\) be a semi-algebraic  
  subset of \(\VF^n(K)\). Let \(L=\alg{K}\).  Then there is a quantifier-free 
  \(\LDIV(K)\)-formula \(\psi(\x)\) with \(\psi(K)=\dX\) and such that 
  \(\psi(K)\) is dense in \(\psi(L)\).
\end{proposition}

\begin{remark}
  Actually, Hong states his result only for \(\aleph_1\)-saturated \(K\), but 
  it is easy to see that the result for general \(K\) follows from this. We 
  add a full proof for convenience.
\end{remark}

\begin{proof}
By Lemma\,\ref{L:strat}, we may assume that \(\dX\) is a (non-empty) valuation-open subset of a  variety \(\dV\) over \(K\). Since \(\dX\) is Zariski-dense in \(\dV\), we 
  may pass to the smooth locus and assume that \(\dV\) is smooth.  By 
  Proposition~\ref{P:manifold}, we may now reduce to the case of 
  \(\dV=\VF^d\). Indeed, as $K$ is separably closed, there exists an open 
  cover of $\dV$ by subvarieties $\dV_i$ and \'etale maps 
  $f_i:\dV_i\rightarrow\VF^d$ defined over $K$ which are local 
  homeomorphisms. As $K=\sep{K}$ and $f_i$ is \'etale, for every 
  $a\in\dV_i(L)$ one has $f_i(a)\in\VF^d(K)$ if and only if $a\in\dV_i(K)$.  
  Now $\VF^d(K)$ is dense in $\VF^d(L)$ by Proposition \ref{prop:SCVF-dense}, 
  and so $\dV_i(K)$ is dense in $\dV_i(L)$.
\end{proof}

\begin{corollary}\label{cor:red to open in affine space}
  Let \(\mK\models\SCVF\) and \(\dX\) be a semi-algebraic  subset of \(\VF^n(K)\).  
  Let \(\dV\) be the Zariski closure of \(\dX\), and let $d=\dim(\dV)$. Then 
  there is a semi-algebraic subset $\dX'$ of $\dX$ and a polynomial map 
  $f:\dX'\rightarrow\VF^d$ defined over $K$ which induces a homeomorphism 
  between $\dX'$ and $\bO^d(K)$.
\end{corollary}

\begin{proof}
  The result follows from the proof of Proposition \ref{P:vdense}.
\end{proof}

\subsection{Elimination of imaginaries in 
\texorpdfstring{\(\SCVF_{p,e}\)}{SCVFpe}}
In this section, \(K\) denotes a sufficiently saturated and homogeneous model 
of \(\SCVF_{p,e}\), and \(L:=\alg{K}\) its algebraic closure, so a model of 
\(\ACVF_{p,p}\). We consider \(L\) in the language \(\LG\), and \(K\) in the 
language \(\LGpe^\lambda\).  In \(K\), \(\eq{\acl}\) etc.  refers to 
\(\SCVF_{p,e}^{\mathrm{eq}}\).

\begin{lemma}\label{L:SCVF-ACVF}
  \begin{enumerate}
    \item Any automorphism of \(K\) (uniquely) lifts to an automorphism of 
      \(L\). In particular, if \(\a,\b\in\bG(K)\) and 
      \(\b\in\dcl_{\ACVF_{p,p}}(\a)\), then \(\b\in\dcl_{\SCVF_{p,e}}(\a)\).
    \item For every tuple \(\a\in\bG(L)\) there is a tuple \(\a'\in\bG(K)\) such that 
      \(\dcl_{\ACVF_{p,p}}(\a)=\dcl_{\ACVF_{p,p}}(\a')\).
    \item In the structure \(K\), finite sets are coded, i.e., for every 
      \(\{\a_1,\ldots,\a_n\}\subseteq\bG(K)\) there is \(\b\in\bG(K)\) which is 
      interdefinable in \(\eq{K}\) with \(\code{\{\a_1,\ldots,\a_n\}}\).
  \end{enumerate}
\end{lemma}

\begin{proof}
  (1) is clear. To prove (2), note that \(\latt_n(K)=\latt_n(L)\) and 
  \(\tor_n(K)=\tor_n(L)\) for every \(n\geq1\). So it is enough to show (2) 
  for elements of the field sort. But for any \(a\in L\) there is \(m\) such 
  that \(a^{p^m}\in K\), and \(a\) and \(a^{p^m}\) are interdefinable.

  (3) By elimination of imaginaries in \(\ACVF_{p,p}\) down to the geometric 
  sorts, it follows that finite subsets of \(\bG(L)\) are coded in \(\bG(L)\).  
  We finish combining (2) and (1).
\end{proof}

\begin{theorem}\label{T:SCVF-EI}
  The theory \(\SCVF_{p,e}\) eliminates imaginaries down to the geometric 
  sorts.
\end{theorem}

\begin{proof}
  Let \(\dX\) be a definable set and  \(A\) contain \(\G(\code{\dX})\).  
  We have to prove that \(\dX\) can be defined over \(A\).  Since, by 
  Lemma~\ref{L:SCVF-ACVF}, finite sets are coded, it is enough to show that 
  \(\dX\) is weakly coded, so we may assume that \(A\) is algebraically 
  closed. Also, because $\VF$ is dominant, we may assume that $\dX$ is a 
  subset of some Cartesian power of $\VF$. Let $K_0:=\VF(A)$. Since \(K_0\) is closed under $\lambda$-functions and relatively algebraically closed in \(K\), the extension $K/K_0$ is regular.

  By Lemma\,\ref{lem:lambda res}, there exists an $n$ such that 
  $\lambda^n(\dX) = \psi(K)$ where $\psi$ is a quantifier free 
  $\LDIV(K)$-formula. As $\lambda$ is injective and $\emptyset$-definable it 
  follows that finding a (weak) code for $\lambda^n(\dX)$ is equivalent to 
  finding one for $\dX$, so we may assume that $\dX=\psi(K)$ for some 
  quantifier free $\LDIV(K)$-formula $\psi(x)$.  Moreover, by 
  Proposition\,\ref{P:vdense}, we may assume that $\dX$ is dense in 
  $\dY:=\psi(L)$. 

  Let $\dV = \Zar{\dX}$. Then $\dV$ is $K_0$-definable, by the existence of a smallest field of definition of $\dV$. Moreover, $\dY\subseteq\dV(L)$. We proceed by induction 
  on $\dim(\dV)$. Since $\dV(K)$ is Zariski-dense in $\dV(L)$ and the extension $K/K_0$ is 
  regular, the (absolute) irreducible components 
  $\dV_1,\ldots,\dV_l$ of $\dV$ are defined over $K_0$. Hence, encoding 
  $\dV_i(K)\cap\dX$ one by one, we may assume that $\dV$ is absolutely 
  irreducible. It obviously suffices to encode $\clvK{\dX}$ and 
  $\clvK{\dX}\setminus\dX$, where $\clvK{\dX}$ denotes the valuative closure 
  of $\dX$ in \(\VF^n(K)\). But $\clvK{\dX}\setminus\dX \subseteq 
  \clvL{\dY}\setminus\dY$, a subset of $\dV(L)$ which has empty interior (in 
  $\dV(L)$), so, by Corollary\,\ref{C:dim}, $\Zar{\clvK{\dX}\setminus\dX} 
  \subseteq \Zar{\clvL{\dY}\setminus\dY}$ is a strict subvariety of $\dV$. By 
  induction $\clvK{\dX}\setminus\dX$ is $A$-definable. It follows that we may 
  assume $\dX$ valuatively closed in $\dV(K)$.

  The set $\widetilde{\dY} = \clvL{\dY} = \clvL{\dX}$ is also definable by a 
  quantifier free $\LDIV(K)$-formula, say by $\widetilde{\psi}$, and one has 
  $\widetilde{\psi}(K) = \clvK{\dX} = \dX$. By elimination of imaginaries in 
  \(\ACVF\) (Fact\,\ref{F:ACVF-EI}), $\widetilde{\dY}$ is definable over some 
  $e\in\bG(\code{\widetilde{\dY}})$ and, by Lemma\,\ref{L:SCVF-ACVF}, we may 
  assume that $e\in \bG(K)$. Clearly $\dX$ is $\LG$-definable over $e$ (in 
  $K$) so there only remains to show that $e \in \code{\dX}$. Let $\sigma$ be 
  an automorphism of $K$ that stabilizes $\dX$ globally, then the (unique) 
  extension of $\sigma$ to $L$ must stabilize $\clvL{\dX} = \widetilde{\dY}$ 
  and hence fixes $e$. So $e\in\code{\dX}$.
\end{proof}

\section{Imaginaries, definable types and dense pairs}\label{S:pairs}

\subsection{Quantifier elimination in dense pairs of valued fields}

Much of the following is inspired by work of Delon \cite{DelDens}.

\subsubsection{The pure field case}

Let $\Lp$ denote the language $\Lrg\cup\{\FF,\lin{n},\flin{n,i}\mid 
n\in\N_{>0}\text{ and }0<i\leq n\}$ where $\FF$ is a new unary predicate, the 
$\lin{n}$ are new $n$-ary predicate symbols and the $\flin{n,i}$ are new 
$n+1$-ary function symbols. Note that, for this section, the field sort will 
be denoted by $\K$ and not by $\VF$, as it is not valued. Let $\Tp$ be the 
$\Lp$-theory of pairs of fields, with $\FF$ defining the smaller field, and 
where $\lin{n}(y_{1},\ldots,y_n)$ holds if and only if the $y_i$ are linearly 
independent over $\FF$ and if $\lin{n}(y)\wedge \neg\lin{n+1}(x,y)$ holds 
then $x =\sum_i \flin{n,i}(x,y)y_i$ where $\flin{n,i}(x,y)\in\FF$ (otherwise, 
set $\flin{n,i}=0$).

When $A$ is a field, we will denote linear disjointness over $A$ by 
$\ldis[A]$.

One can easily check the following two facts.

\begin{lemma}\label{lem:crit substr}
  Let $M\models\Tp$ and $A\subseteq M$. Then $A\substr M$ if and only if $A$ 
  is a subring, $\FF(A)$ is a field and $A\ldis[\FF(A)]\FF(M)$.
\end{lemma}

\begin{proof}
  First of all, $A$ is an $\Lrg$-substructure of $M$ if and only if it is a 
  subring.

  Let us now assume $A$ is closed under the functions $\flin{n,i}$. Note 
  that, if $x\in\FF$, $\flin{1,1}(1,x) = x^{-1}$ and hence $\FF(A)$ is a 
  field. Now, let $a\in A$ be a tuple. If $a$ is not linearly independent 
  over $\FF(M)$, we may assume that $a_{0} = \sum_{i>0} c_i a_i$ where the 
  $(a_i)_{i>0}$ are linearly independent over $\FF(M)$. Then $c_i = 
  \flin{n,i}(a)\in \FF(A)$ and hence the tuple $a$ is not linearly 
  independent over $\FF(A)$. We have just proved that $A\ldis[\FF(A)]\FF(M)$.

  Conversely, assume that $\FF(A)$ is a field and $A\ldis[\FF(A)]\FF(M)$.  We 
  have to show that $A$ is closed under the $\flin{n,i}$ functions.  Let $a$ 
  be a tuple in $A$ such that $(a_i)_{i>0}$ is linearly independent over 
  $\FF(M)$ and $a_0 = \sum_{i>0} \flin{n,i}(a)a_i$. By hypothesis, the tuple 
  $a$ is not linearly independent over $\FF(A)$ either and (because $\FF(A)$ 
  is a field) there exist elements $c_i\in\FF(A)$ such that $a_{0} = \sum_{i>0} c_i 
  a_i$. Since the $(a_i)_{i>0}$ are linearly independent, we have 
  $\flin{n,i}(a) = c_i\in\FF(A)$.
\end{proof}

\begin{lemma}\label{lem:crit iso}
  Let $M_i\models\Tp$, $A_i\substr M_i$ for $i=1,2$, $f : A_1\to A_2$ an 
  $\Lrg$-isomorphism such that $f(\FF(A_1))=\FF(A_2)$. Then $f$ is in fact an 
  $\Lp$-isomorphism.
\end{lemma}

\begin{proof}
  We have to check that $f$ respects the predicates $\lin{n}$ and the 
  functions $\flin{n,i}$. First let $a\in A_1$ be a tuple. The tuple $a$ is 
  linearly dependent over $\FF(M_1)$ if and only if it is linearly dependent 
  over $\FF(A_1)$, i.e., there exist $\lambda_i\in\FF(A_1)$ such that 
  $\sum_i\lambda_i a_i = 0$. Equivalently, $f(\sum_i\lambda_i a_i) = \sum_i 
  f(\lambda_i)f(a_i) = 0$ and the tuple $f(a)$ is linearly dependent over 
  $\FF(M_2)$. By symmetry, we also have that if $f(a)$ linearly dependent 
  over $\FF(M_2)$, then $a$ is linearly dependent over $\FF(M_1)$. Hence $f$ 
  respects $\lin{n}$.

  Let us now assume that $a$ is linearly independent over $\FF(M_1)$ and that 
  $c = \sum_i \flin{n,i}(c,a)a_i$. Then $f(c) = f(\sum_i \flin{n,i}(c,a)a_i) 
  = \sum_i f(\flin{n,i}(c,a))f(a_i)$. Moreover the tuple $(c,a)$ is linearly 
  dependent over $\FF(M_1)$ but $a$ is not and hence $(f(c),f(a))$ is 
  linearly dependent over $\FF(M_2)$ but $f(a)$ is not. Therefore $f(c) = 
  \sum_i\flin{n,i}\left(f(c),f(a)\right) f(a_i)$ and by uniqueness of the 
  coefficient in a decomposition of $f(c)$ in the basis $f(a)$, we obtain 
  that $f(\flin{n,i}(c,a)) = \flin{n,i}(f(c),f(a))$.
\end{proof}

\begin{lemma}\label{lem:frac pure}
  Let $M_i\models\Tp$, $A_i\substr M_i$ for $i=1,2$, $f : A_1\to A_2$ an 
  $\Lp$-isomorphism. Then $\Frac{A_i}\substr M_i$, $\FF(\Frac{A_i}) = 
  \FF(A_i)$ and $f$ extends to a unique $\Lp$-isomorphism between 
  $\Frac{A_1}$ and $\Frac{A_2}$.
\end{lemma}

\begin{proof}
  One checks, by clearing the denominators, that 
  $\Frac{A_i}\ldis[\FF(A_i)]\FF(M_i)$ and thus that $\FF(\Frac{A_i})$ = 
  $\FF(A_i)$. Lemmas\,\ref{lem:crit substr} and \ref{lem:crit iso} now allow 
  us to conclude.
\end{proof}

\begin{lemma}\label{lem:small field pure}
  Let $M_i\models\Tp$, $A_i\substr M_i$, $\FF(A_i)\substr[\Lrg] 
  C_i\substr[\Lrg]\FF(M_i)$ for $i=1,2$, $f : A_1\to A_2$ an 
  $\Lp$-isomorphism and $g: C_1\to C_2$ an $\Lrg$-isomorphism such that 
  $\restr{g}{\FF(A_1)} = \restr{f}{\FF(A_1)}$. Assume that $A_i$ and $C_i$ 
  are fields then $A_i C_i \substr M_i$, $\FF(A_i C_i) = C_i$ and there is a 
  unique $\Lp$-isomorphism $h:A_1 C_1\to A_2 C_2$ extending $f$ and $g$.
\end{lemma}

\begin{proof}
  As $A_i\ldis[\FF(A_i)]\FF(M_i)$ and $C_i\substr\FF(M_i)$, we have $A_i C_i 
  \ldis[C_i]\FF(M_i)$. It follows that $\FF(A_i C_i) = C_i$ and, by 
  Lemma\,\ref{lem:crit substr}, that $A_i C_i \substr M_i$.  As 
  $A_i\ldis[\FF(A_i)]\FF(M_i)$, it follows that $A_i C_i$ is isomorphic as a 
  ring to $A_i \tensor[\FF(A_i)] C_i$ and hence there exists a unique ring 
  isomorphism $h:A_1 C_1\to A_2 C_2$ extending $f$ and $g$.  By 
  Lemma\,\ref{lem:crit iso}, $h$ is in fact an $\Lp$-isomorphism.
\end{proof}

\begin{lemma}\label{lem:alg pure}
  Let $M_i\models\Tp$, $A_i\substr M_i$ for $i=1,2$ and $f : A_1\to A_2$ an 
  $\Lp$-isomorphism. Assume $A_i$ is a field and the extension 
  $\FF(A_i)\subseteq \FF(M_i)$ is regular then $\alg{A_i}\cap M_i\substr 
  M_i$, $\FF(\alg{A_i}\cap M_i) = \FF(A_i)$ and any field isomorphism 
  $g:\alg{A_1}\cap M_1\to\alg{A_2}\cap M_2$ extending $f$ is an 
  $\Lp$-isomorphism.
\end{lemma}

\begin{proof}
  We have $A_i\ldis[\FF(A_i)]\FF(M_i)$ and $\FF(A_i)\subseteq \FF(M_i)$ 
  regular hence $A_i\subseteq A_i\FF(M_i)$ is also regular, that is, 
  $\alg{A_i}\ldis[A_i]A_i\FF(M_i)$. By transitivity of linear disjointness, 
  we obtain that $\alg{A_i}\ldis[\FF(A_i)]\FF(M_i)$ and hence $\alg{A_i}\cap 
  M_i\ldis[\FF(A_i)]\FF(M_i)$. In particular, we have $\FF(\alg{A_i}\cap M_i) 
  = \FF(A_i)$ and we may conclude using Lemmas \ref{lem:crit substr} and 
  \ref{lem:crit iso}.
\end{proof}

\begin{lemma}\label{lem:trans pure}
  Let $M_i\models\Tp$, $A_i\substr M_i$, $c_i\in M_i$ for $i=1,2$ and $f : 
  A_1\to A_2$ an $\Lp$-isomorphism. Assume $A_i$ is a field and $c_i$ 
  transcendental over $A_i\FF(M_i)$ then $A_i(c_i)\substr M_i$, 
  $\FF(A_i(c_i)) = \FF(A_i)$ and there exists an $\Lp$-isomorphism 
  $g:A_1(c_1)\to A_2(c_2)$ extending $f$ and sending $c_1$ to $c_2$.
\end{lemma}

\begin{proof}
  We have that $f$ extends to a ring isomorphism $g$ on $A_1(c_1)$ sending 
  $c_1$ to $c_2$. Moreover, $A_i(c_i)$ is algebraically independent from 
  $A_i\FF(M_i)$ over $A_i$. As $A_i\subseteq A_i(c_i)$ is regular, $A_i(c_i) 
  \ldis[A_i] A_i\FF(M_i)$. Since $A_i\ldis[\FF(A_i)]\FF(M_i)$, by 
  transitivity, it follows that $A_i(c_i)\ldis[\FF(A_i)]\FF(M_i)$. In 
  particular $\FF(A_i(c_i)) = \FF(A_i)$ and, by Lemma\,\ref{lem:crit iso}, 
  $g$ is in fact an $\Lp$-isomorphism.
\end{proof}

We can now prove a slightly improved version of \cite[Theorem\,1]{DelDens}.  
Let $\alg{\Tp} := \Tp\cup\ACF\cup\{[\K:\FF] \geq n\mid n\in\N\}$.

\begin{theorem}\label{thm:EQ pure}
  The theory $\alg{\Tp}$ resplendently eliminates quantifiers relative to 
  $\FF$; that is, for every language $\tL\supseteq\Lrg$ and every 
  $\tL$-theory $\tT$ which eliminates quantifiers, the $\Lp\cup\tL$-theory 
  $\Tpp:=\alg{\Tp}\cup \tT^{\FF}$ eliminates quantifiers where $\tT^{\FF}$ is 
  the relativization of $\tT$ to $\FF$.
\end{theorem}

\begin{proof}
  Let us denote by $\Lpp$ the language $\Lp\cup\tL$. Note that an 
  $\Lpp$-isomorphism is an $\Lp$-isomorphism that restricts to an 
  $\tL$-isomorphism on $\FF$.

  Let $M$ and $N$ be models of $\Tpp$, $A\substr M$ and $f:A\to N$ an 
  $\Lpp$-embedding. Assume that $N$ is $\card{M}^+$-saturated. We have to 
  show that $f$ extends to $M$. By Lemma\,\ref{lem:frac pure}, we may assume 
  that $A$ is a field. Since $\FF(\Frac{A}) = \FF(A)$, this new embedding is 
  an $\tL$-embedding. By quantifier elimination in $\tT$, we can extend 
  $\restr{f}{\FF(A)}$ to $g:\FF(M)\to\FF(N)$ and by Lemma\,\ref{lem:small field 
  pure}, we may assume that $\FF(M)\subseteq A$. As $\tT$ eliminates 
  quantifiers, $f(\FF(M))\subsel_{\tL} \FF(N)$ and this extension is regular.  
  Applying Lemma\,\ref{lem:alg pure}, we may assume that $A$ is algebraically 
  closed.

  \begin{claim}\label{claim:tr deg pure}
    The transcendence degree of $N$ over $\FF(N)$ is larger than $\card{M}$.
  \end{claim}

  \begin{proof}
    By compactness and saturation this follows from the fact that $N$ is an 
    infinite extension of $\FF(N)$.
    %We have to find $(x_{\alpha})_{\alpha\in\card{M}^+}$ such that the 
    %following holds for all tuple $\alpha\in\kappa$ and $d$, $\forall 
    %a_I\in\FF\,\sum_{\card{I}\leq d} a_Ix_\alpha^I\neq 0$.  By compactness, 
    %it %suffices to show that one of these formulas is satisfiable in $N$.  
    %But because $N$ is an infinite extension of $\FF(N)$, one can find 
    %$x_i\in N$ whose degree over $\FF(N)(x_j:0\leq j < i)$ is greater than 
    %$d$.
  \end{proof}

  Let $c\in M$ be transcendental over $A$ and $d\in N$ be transcendental over 
  $f(A)\FF(N)$. Then, by Lemma\,\ref{lem:trans pure}, $f$ extends to an 
  $\Lp$-embedding $g:A(c)\to N$ sending $c$ to $d$. Moreover, $\restr{g}{\FF} 
  = \restr{f}{\FF}$ and hence $g$ is also an $\tL$-embedding. Now, by 
  Lemma\,\ref{lem:alg pure}, $g$ extends to an $\Lpp$-embedding on 
  $\alg{A(c)}$. Repeating this last step sufficiently many times, we obtain 
  an $\Lpp$-embedding of $M$ into $N$.
\end{proof}

\subsubsection{The valued case}

We now want to extend the previous results to the setting of dense pairs of 
valued fields.

Let $\Lpv := \Lp\cup\LRV$. We will consider the $\Lpv$-theory \[\Tpv := 
\Tp\cup\{\text{$\VF$ is a valued field and $\FF\subseteq\VF$ is dense}\}.\]

In any model $M\models \Tpv$, by density of the pair $\FF(M)\subseteq\VF(M)$, 
we have $\rv(\VF(M)) = \rv(\FF(M))$. We define $\FFRV := \FF\cup\RV$.

\begin{lemma}\label{lem:frac valued}
  Let $M_i\models\Tpv$, $A_i\substr M_i$ for $i=1,2$ and $f : A_1\to A_2$ an 
  $\Lpv$-isomorphism. Then $f$ extends to an $\Lpv$-isomorphism between 
  $A_1(\Frac{\VF(A_1)})$ and $A_2(\Frac{\VF(A_2}))$.
\end{lemma}

\begin{proof}
  By quantifier elimination in $\LRV$, $f$ extends to an 
  $\LRV$-isomorphism $g:A_1(\Frac{\VF(A_1)}) \to A_2(\Frac{\VF(A_2)})$.  By 
  Lemma\,\ref{lem:frac pure}, $\restr{g}{\VF}$ is an $\Lp$-isomorphism.
\end{proof}

\begin{lemma}\label{lem:small field valued}
  Let $M_i\models\Tpv$, $A_i\substr M_i$, $\FFRV(A_i)\substr[\LRV] 
  C_i\substr[\LRV]\FFRV(M_i)$ for $i=1,2$, $f : A_1\to A_2$ an 
  $\Lpv$-isomorphism and $g: C_1\to C_2$ an $\LRV$-isomorphism such that 
  $\restr{g}{\FFRV(A_1)} = \restr{f}{\FFRV(A_1)}$. Moreover, assume that for 
  all $a\in \VF(A_i)\setminus\VF(C_i)$ there is a Cauchy sequence 
  $(c_{\alpha}^a)$ in $\VF(C_i)$ converging to $a$ (in $\VF(A_i(C_i))$) and 
  such that $g(c_{\alpha}^a) = c_{\alpha}^{f(a)}$. Assume that $\VF(A_i)$ and 
  $\VF(C_i)$ are fields. Then there exists a unique $\Lpv$-isomorphism 
  $h:A_1(C_1)\to A_2(C_2)$ extending $f$ and $g$.
\end{lemma}

\begin{proof}
  Our assumptions imply that $\VF(C_i)$ is dense in $\VF(A_i(C_i))$.  By 
  Lemma\,\ref{lem:small field pure}, there exists a unique $\Lp$-isomorphism 
  $h : \VF(A_1(C_1))\to \VF(A_2(C_2))$ extending $\restr{f}{\VF}$ and 
  $\restr{g}{\VF}$. We extend it to \(\RV\) by setting $\restr{h}{\RV} := 
  \restr{g}{\RV}$. Note that, by density, $\RV(A_i(C_i))=\RV(C_i)$.

  We have to check that $h$ preserves $\rv$. Let us assume that all 
  $(c_\alpha^a)$ are indexed by the same ordinal (the cofinality of 
  \(\val(\VF(C_i)^\star)\)). Let $P(\ol{X})\in C_1[\ol{X}]$ and $a\in 
  \VF(A_1)^{\card{\ol{X}}}$. If $P(a) = 0$ then $h(P(a)) = 0$ and 
  $\rv(h(P(a)))=\infty=h(\rv(P(a)))$. Thus, we may assume that $P(a)\neq 0$ 
  and $P^g(f(a)) \neq 0$. Since $\val(a-c_{\alpha}^a)$ is cofinal in 
  $\val(\VF(A_1(C_1)))$ and $P(c_{\alpha}^{a}) = P(a) + \sum_{I\neq 0} 
  (c_{\alpha}^{a}-a)^I P_I(a)$, for large enough $\alpha$, one has $\rv(P(a)) 
  = \rv(P(c_{\alpha}^{a}))$. As $g$ is an $\LRV$-isomorphism, 
  $g(\rv(P(c_{\alpha}^{a}))) = \rv(P^g(c_{\alpha}^{f(a)}))$. Similarly, 
  $\rv(P^g(c_{\alpha}^{f(a)})) = \rv(P^g(f(a))) = \rv(h(P(a)))$.  Hence 
  $h(\rv(P(a))) = g(\rv(P(c_{\alpha}^{a}))) = \rv(h(P(a)))$.
\end{proof}

\begin{corollary}\label{cor:small field valued}
  Let $M_i\models\Tpv$, $A_i\substr M_i$, $\FFRV(A_i)\substr[\LRV] 
  C_i\substr[\LRV]\FFRV(M_i)$ for $i=1,2$, $f : A_1\to A_2$ an 
  $\Lpv$-isomorphism and $g: C_1\to C_2$ an $\LRV$-isomorphism such that 
  $\restr{g}{\FFRV(A_1)} = \restr{f}{\FFRV(A_1)}$. 

  Assume that $\FF(A_i)$ is dense in $\VF(A_i)$, $\val(\VF(A_i)^{\star})$ is a 
  cofinal subset of $\val(\VF(A_i(C_i))^{\star})$ and that $\VF(A_i)$ and $\VF(C_i)$ 
  are fields. Then there exists a unique $\Lpv$-isomorphism $h:A_1(C_1)\to 
  A_2(C_2)$ extending $f$ and $g$.
\end{corollary}

\begin{proof}
  This is an immediate consequence of Lemma\,\ref{lem:small field valued}.  
  Indeed, for all $a\in A_1\setminus C_1$, let $c_\alpha^a$ be Cauchy 
  sequence in $\FF(A_1)$ converging to $a$ in $\VF(A_1)$ (and hence in 
  $\VF(A_1(C_1))$ by cofinality) and define $c_\alpha^{f(a)}$ to be equal to 
  $f(c_\alpha^{a})$. Now, Lemma\,\ref{lem:small field valued} applies.
\end{proof}

\begin{comment}
  \begin{lemma}\label{lem:alg valued}
    Let $M_i\models\Tpv$, $A_i\substr M_i$ for $i=1,2$ and $f : A_1\to A_2$ 
    an $\Lpv$-isomorphism. Assume $\VF(A_i)$ is a field, the extension 
    $\FF(A_i)\subseteq \FF(M_i)$ is regular and $\FF(A_i)$ is dense in 
    $\alg{\VF(A_i)}$. Then any extension of $f$ to an $\LRV$-isomorphism $g : 
    A_1(\alg{\VF(A_1)}\cap M_1) \to A_2(\alg{\VF(A_2)}\cap M_2)$ is an 
    $\Lpv$-isomorphism.
  \end{lemma}

  \begin{proof}
    We only have to show that $\restr{g}{\VF}$ is an $\Lp$-isomorphism.  But 
    that follows from Lemma\,\ref{lem:alg pure}.
  \end{proof}

  \begin{lemma}\label{lem:trans valued}
    Let $M_i\models\Tpv$, $A_i\substr M_i$, $c_i\in M_i$, and $f : A_1\to 
    A_2$ an $\Lpv$-isomorphism. Assume $A_i$ is a field, the extension 
    $\FF(A_i)\subseteq \FF(M_i)$ is regular and $c_i$ transcendental over 
    $A_i\FF(M_i)$. Then any $\LRV$-isomorphism $g:A_1(c_1)\to A_2(c_2)$ 
    extending $f$ is an $\Lpv$-isomorphism.
  \end{lemma}

  \begin{proof}
    We only have to show that $\restr{g}{\VF}$ is an $\Lp$-isomorphism.  But 
    that follows from Lemma\,\ref{lem:trans pure}.
  \end{proof}
\end{comment}

Let $\L$ be an $\RV$-enrichment of $\LRV$ and $T$ be an $\L$-theory of valued 
fields that eliminates quantifiers. The main two examples are $\ACVF$ and 
theories of equicharacteristic zero Henselian fields Morleyized on $\RV$.  As 
always, the present results remain true in mixed characteristic provided 
$\LRV$ is understood as the language of higher order leading terms, i.e., with 
sorts for $\RV_n := \VF/(1+n\boldmax)$ for all $n\in\N$. Let $\tL\supseteq\L$ be a 
language and $\tT$ be an $\tL$-theory which eliminates quantifiers. We will 
now consider the theory \[\Tppv := \Tpv\cup T\cup\tT^{\FFRV}\cup\{[\VF:\FF] 
\geq n\mid n\in\N\}.\]

We will be denoting by $\Balls$ the (interpretable) set of open balls.

\begin{theorem}\label{thm:EQ valued}
  Assume that:
  \begin{enumerate}
    \item For all $M\models\Tppv$, $A\leq\FF(M)$ and tuples $b_1$, 
      $b_2\in\Balls(M)$: If $b_1\equiv_{\L(A)}^{M}b_2$ then 
      $b_1\equiv_{\tL(A)}^{\FF(M)}b_2$.
    \item For every $A\leq \FF(M)\models\tT$ with $M$ sufficiently saturated 
      with respect to $|A|$, $C_A:= \{x\in \VF(M)\mid$ the \(\tL\)-structure 
      and the \(\L\)-structure generated by \(Ax\) are equal\(\}\) is dense 
      in $\VF(M)$.
  \end{enumerate}
  Then $\Tppv$ eliminates quantifiers. In particular, if $M\models\Tppv$, then $\FF(M)$ is stably embedded in 
  $M$ with induced structure given by $\tL$.
\end{theorem}

\begin{proof}
  Let $M$ and $N$ be models of $\Tppv$, $A\substr M$ and $f:A\to N$ an 
  $\Lppv$-embedding, where $\Lppv:=\Lpv\cup\tL$. Assume that $N$ is 
  $\card{M}^+$-saturated. We have to show that $f$ extends to $M$. By 
  Lemma\,\ref{lem:frac valued}, we may assume that $\VF(A)$ is a field.

  \begin{claim}
    At the cost of enlarging $M$ (without changing its cardinality) and 
    $A\leq M$, we may assume that $\FF(A)$ is dense in $\VF(A)$ and 
    $\val(\VF(A)^\star)$ is cofinal in $\val(M^\star)$.
  \end{claim}

  \begin{proof}
    Let $M'$ be some sufficiently saturated extension of $M$.  Let $\kappa := 
    \card{\VF(A)}$ and $(a_i)_{i\in\kappa}$ be an enumeration of $\VF(A)$.

    By induction on $(n,i)\in\omega\times\kappa$ (ordered lexicographically), 
    we will construct
    \begin{itemize}
      \item a sequence $(e_{n,i})_{n\in\omega,i\in\kappa}$ in $\FF(M')$;
      \item an increasing chain $(M_{n,i})_{n\in\omega,i\in\kappa}$ of 
        elementary substructures of $M'$ containing $M$;
      \item an increasing chain of $\Lppv$-embeddings $f_{n,i}:A_{n,i}\to N$ 
        extending $f$, where $A_{n,i}$ is the structure generated by 
        $A\cup\{e_{m,j}\,|\,(m,j)\leq(n,i)\}$,
    \end{itemize}

    satisfying the following conditions for all $(n,i)\in\omega\times\kappa$:
    \begin{enumerate}
      \item[(i)] $|M_{n,i}|=|M|$;
      \item[(ii)] $e_{n,i}\in M_{n,i}$;
      \item[(iii)] $\val(e_{n,i} - a_{i}) > \val(M_{<(n,i)}^{\star})$, where 
        $M_{<(n,i)}:=M\cup \bigcup_{(n',i')<(n,i)}M_{n',i'}$.
    \end{enumerate}

    We will denote the image of $f$ by $C$, and that of $f_{n,i}$ by 
    $C_{n,i}$.

    Let us first construct $e_{0,0}$, $M_{0,0}$ and $f_{0,0}$. By saturation 
    of $M'$, and Hypothesis 2, we find $e_{0,0}\in\FF(M')$ such that 
    $\val(e_{0,0}-a_0)>\vg(M)\setminus\{\infty\}$ and the \(\tL\)-structure 
    and the \(\L\)-structure generated by \(\FF(A)e_{0,0}\) are equal. We 
    also find an elementary submodel $M_{0,0}$ of $M'$ of cardinality $|M|$ 
    containing $M(e_{0,0})$. Let us now show that we may extend $f$ to 
    $e_{0,0}$.
 
    Note that Hypothesis 1 implies that if \(\eta:A\to N\) is an \(\tL\)-isomorphism between subsets of \(\FF\) and \(\rho\) is an extension of \(\eta\) to some tuple of balls \(b\), which is an $\L$-embedding, then there exists an \(\tL\)-embedding extending \(\eta\) and sending \(b\) to \(\rho(b)\). Indeed, by quantifier elimination in \(\tL\), \(\eta^{-1}\) can be extended to an \(\tL\)-embedding \(i\) defined at \(\rho(b)\). Then \(i\circ \rho\) fixes \(A\). So applying Hypothesis 1 to \(b\) and \(i(\rho(b))\), we get that there exists an \(\tL\)-embedding \(j\) fixing \(A\) and sending \(b\) to \(i(\rho(b))\). Then \(h := i^{-1}\circ j\) is an \(\tL\)-embedding extending \(\eta\) and sending \(b\) to \(\rho(b)\).

    Let us now come back to our previous notations. By quantifier elimination in $T$, $f$ extends to an elementary 
    $\L$-embedding $g:M_{0,0}\to N$. By the previous paragraph and  quantifier elimination in \(\tT\), there exists an 
    $\tL$-embedding $h : \FF(M_{0,0})\to \FF(N)$ extending $\restr{f}{\FF}$ 
    such that $\restr{h}{\Balls(M_{0,0})} = \restr{g}{\Balls(M_{0,0})}$.  Let 
    $c_{0,0} = h(e_{0,0}) \in\FF(N)$. By construction, we have that 
    $\tp_{\tL}(c_{0,0}/\FF(C)) = f_{\star}\tp_{\tL}(e_{0,0}/\FF(A))$.

    Let $d\in M$, $\gamma := \rv(e_{0,0} - d)$ and $b := \{x\mid \rv(x - d) = 
    \gamma\}$. We have $e_{0,0}\in b \in\Balls(M_{0,0})$ and hence 
    $c_{0,0}\in h(b) = g(b) = \{x\mid \rv(x-g(d)) = g(\gamma)\}$. It follows 
    that $\rv(c_{0,0} - g(d)) = g(\rv(e_{0,0} - d))$. In particular 
    $\val(c_{0,0}-f(a_{0})) > \vg(M)\setminus\{\infty\}$ and hence for all non-zero polynomials $P 
    = \sum_i g(d_i) X^i\in\VF(g(M))[X]$, letting $i_0=\min\{i\mid g(d_i)\neq0\}$, we have 
    \begin{multline*}
      \rv(P(c_{0,0}-f(a_0)))=\rv(g(d_{i_0}))\cdot \rv(c_{0,0}-f(a_0))^{i_0} \\
=g(\rv(d_{i_0}(e_{0,0}-a_0)^{i_0})   = g(\rv(P(e_{0,0} - a_0)))
    \end{multline*}
    It follows (since $g$ is an $\L$-isomorphism) that 
    $\tp_{\L}(c_{0,0}/g(M)) = g_{\star}\tp_{\L}(e_{0,0}/M)$. In particular, 
    $\tp_{\L}(c_{0,0}/C) = f_{\star}\tp_{\L}(e_{0,0}/A)$.

    Let $f_{0,0}$ be an $\L$-embedding extending $f$ to \(A(e_{0,0})\)and 
    sending $e_{0,0}$ to $c_{0,0}$. Then $\restr{f_{0,0}}{\FF}$ is an 
    $\tL$-embedding (remember that the \(\tL\)-structure and the 
    \(\L\)-structure generated by \(\FF(A)e_{0,0}\) are equal). Finally, by 
    Lemma\,\ref{lem:small field pure}, $f_{0,0}$ is also an $\Lp$-embedding.

    Now let $(n,i)>(0,0)$ be given, and assume that $e_{m,j}$, $M_{m,j}$ as 
    well as $f_{m,j}$ have been constructed for all $(m,j)<(n,i)$ satisfying 
    the above.  We may then construct $e_{n,i}$, $M_{n,i}$ and $f_{n,i}$ in 
    the exact same way. In the argument above, it is enough to replace $M$ by 
    $M_{<(n,i)}$, $A$ by $A_{<(n,i)}=\bigcup_{(m,j)<(n,i)}A_{m,j}$ and $f$ by 
    $f_{<(n,i)}=\bigcup_{(m,j)<(n,i)}f_{m,j}$. 

    We define $f_{<(\omega,\kappa)}$, $M_{<(\omega,\kappa)}$ and 
    $A_{<(\omega,\kappa)}\leq M_{<(\omega,\kappa)}$ similarly. It is easy to 
    check that $\val(\FF(A_{<(\omega,\kappa)}^\star))$ is cofinal in 
    $\val(M_{<(\omega,\kappa)}^{\star})$ and, as the 
    $(e_{n,i})_{n\in\omega}$ are Cauchy sequences whose limit is $a_i$ in 
    $\VF(A_{<(\omega,\kappa)})$, that $\FF(A_{<(\omega,\kappa)})$ is dense in 
    $\VF(A_{<(\omega,\kappa)})$.
  \end{proof}

  We can now apply Corollary\,\ref{cor:small field valued} to extend $f$ to 
  all of $\FF(M)$ and we may assume that $\FF(A) = \FF(M)$. We may now extend 
  $f$ to the relative algebraic closure of $A$ in $M$ using 
  Lemma\,\ref{lem:alg pure} and quantifier elimination in $T$.

  \begin{claim}\label{claim:tr ball}
    Any ball from $N$ has transcendence degree larger or equal to 
    $\card{M}^+$ over $\FF(N)$.
  \end{claim}

  \begin{proof}
    It suffices to prove the claim for $\O$ and, in that case, it is an easy 
    consequence of Claim\,\ref{claim:tr deg pure}.
  \end{proof}

  Now let $a\in M\setminus A$, then $a\in\VF(M)$, and $a$ is transcendental 
  over $\VF(A)$. Let $(a_\alpha)$ be a Cauchy-sequence in $A$ converging to 
  $a$. Note that $\val(a - a_\alpha)$ is cofinal in $\vg(M)$. Let $b$ be a 
  ball in $N$ that only contains pseudo-limits of the sequence 
  $(f(a_\alpha))$. By Claim\,\ref{claim:tr ball}, we can find a point $c\in 
  b$ which is transcendental over $\VF(f(A))\FF(N)$.  Let $P\in\VF(A)[X]$, 
  then $v(P(a)) < \infty$ and for large enough $\alpha$, $\rv(P(a)) = 
  \rv(P(a_\alpha))$ (cf. the proof of Lemma\,\ref{lem:small field valued}).  
  Similarly, $\rv(P^f(c)) = \rv(P^f(f(a_\alpha))) = f(\rv(P(a)))$. It follows 
  that $f$ can be extended to an $\L$-embedding sending $a$ to $c$. By 
  Lemma\,\ref{lem:trans pure}, this extension is an $\Lppv$-embedding.  
  Repeating these last two steps we can extend $f$ to $M$.
\end{proof}

\begin{remark}
  The hypotheses of Theorem\,\ref{thm:EQ valued} are verified in the two 
  following cases:
  \begin{itemize}
    \item If $\tT = T$, Hypothesis 1 follows immediately from the fact that 
      $\tL = \L$ and the fact that $\FF(M)\subsel_{\L} M$.  
      Hypothesis 2 is trivial in this case as $\tL = \L$. The previous result therefore applies to dense elementary extensions of characteristic zero Henselian fields.
    \item Let $T = \ACVF_{p,p}$ and $\tT = \SCVH_{p,e}$ (or $\SCVF_{p,e}$, 
      respectively). Hypothesis 1 follows from the fact that if 
      $b_1\equiv_{\LRV(A)}^M b_2$  for $b_1,b_2\in\Balls(M)=\Balls(\FF(M))$ and $A\leq\FF(M)$ then $b_1\equiv_{\LG(A)}^M b_2$ (since the additional sorts in \(\LG\) are interpretable) and thus $b_1\equiv_{\LGpe^D(A)}^{\FF(M)} b_2$ (similarly for $\LGpe^\lambda$) by Corollary \ref{cor:purely-geometric}.

      Hypothesis 2 follows from the fact that for all $A$, 
      $\ppowi{\FF(M)}\subseteq C_A$ and $\ppowi{\FF(M)}$ is dense in 
      $\FF(M)$, a field which is dense in $\VF(M)$ by assumption.
  \end{itemize}
\end{remark}

\subsection{Dense pairs \texorpdfstring{\((\ACVF,\SCVF)\)}{(ACVF,SCVF)}}

We will now focus on the case $T = \ACVF_{p,p}^{\G}$ and $\tT = 
\SCVH_{p,e}^{\G}$ (there are similar corollaries in the case $\tT = 
\SCVF_{p,e}^{\G}$). Note that the geometric language does not exactly fit in the 
setting of the previous section (since there are additional geometric sorts), 
but for the results that we are giving here the precise language in which we 
are working is completely immaterial.

\begin{corollary}\label{cor:st emb pair}
  In models of $\Tppv$, $\vg$ is stably embedded and a pure divisible ordered 
  Abelian group and $\k$ is stably embedded and a pure algebraically closed 
  field.
\end{corollary}

\begin{proof}
  It follows from Theorem\,\ref{thm:EQ valued}, that $\RV$ is stably embedded 
  in models of $\Tppv$ and that its structure is the structure induced by 
  $\ACVF$. Stable embeddedness and purity of $\vg$ and $\k$ now follow from 
  the equivalent result in $\ACVF$.
\end{proof}

\begin{corollary}\label{cor:complete th pair}
  The theory $\Tppv$ is complete.
\end{corollary}

\begin{proof}
  The field $\F_p$ can be embedded, as a subset of $\FF$, in every model of 
  $\Tppv$. Completeness follows from Theorem\,\ref{thm:EQ valued}.
\end{proof}

\begin{lemma}\label{lem:max complete SCVF}
  Let \(K\models \SCVF\). The following are equivalent:
  \begin{enumerate}[(i)]
    \item \(K\) does not admit any separable immediate valued field 
      extension, i.e., \(K\) is separably maximally complete;
    \item \(\alg{K}\) does not admit any immediate valued field extension, 
      i.e., \(\alg{K}\) is maximally complete.
  \end{enumerate}
\end{lemma}

\begin{proof}
  Let us first prove that (i) implies (ii). Let $\alg{K} \subseteq L$ be a 
  proper immediate extension. We may assume that $L = \alg{K}(t)$, where $t$ 
  is a transcendental element over $\alg{K}$. But then \(K(t)/K\) is 
  immediate (and separable of course).

  Conversely,  assume that \(K\) has a proper immediate separable extension.  
  As \(K\) is separably closed, this cannot be an algebraic extension and we 
  may thus assume it is of the form \(K(t)\) with $t$ transcendental. As 
  $\vg(K)=\vg(K(t))$ is divisible and $\rf(K)=\rf(K(t))$ is algebraically 
  closed, the algebraic extension $\alg{K}(t)/K(t)$ is immediate, and so 
  $\alg{K}(t)$ is a proper immediate extension of $\alg{K}$.
\end{proof}

\begin{proposition}\label{prop:complete model pair}
  The theory $\Tppv$ admits a model $M$ such that $\vg(M)=\R$ and $\VF(M)$ is 
  maximally complete.
\end{proposition}

\begin{proof}
  Let $(x_{\alpha})_{\alpha\in2^{\aleph_{0}}}$ be a linear basis of $\R$ over 
  $\Q$.
  Let $L_{0} := \F_{p}(t_{i})_{1\leq i\leq e}$ be trivially valued and $L_{1} 
  := L_{0}(x_{\alpha}^{p^{-\infty}}:\alpha\in2^{\aleph_{0}})$ be equipped 
  with the unique valuation such that $\val(x_{\alpha}) = x_{\alpha}$. Then 
  $L_{2} := \sep{L_{1}}$ is separably closed of characteristic $p$ and Ershov invariant $e$ and so is any separably maximally complete extension 
  $L\supseteq L_1$.  The pair $(\alg{L},L)$ is dense and $\val(\alg{L}) = 
  \Q\otimes\langle x_{\alpha}\rangle_{\alpha\in2^{\aleph_{0}}} = \R$. Then 
  $L$ can be endowed with Hasse derivations so that 
  $(\alg{L},L)\models\Tppv$.

  Moreover, by Lemma\,\ref{lem:max complete SCVF}, $\alg{L}$ is maximally 
  complete.
\end{proof}

\subsection{Imaginaries in \texorpdfstring{$\SCVH_{p,e}$}{SCVHpe}}
We begin with a review of some results from~\cite{RidVDF,RidSimNIP}. Consider 
an arbitrary complete theory \(T\), and a fixed universal domain \(\mM\).  As 
usual, for a definable set \(\dX\) over parameters, we denote by 
\(\acl(\code{\dX})\) the set of elements of \(\mM\) whose orbit under the 
(set-wise) stabilizer of \(\dX(\mM)\) in \(\Aut(\mM)\) is finite.  We will 
use the following criterion for elimination of imaginaries:

\begin{proposition}[{\cite[Proposition\,10.1]{RidVDF}}]\label{prop:EICrit}
  Assume that every non-empty definable set \(\dX\) in a theory \(T\) 
  contains an \(\acl(\code{\dX})\)-invariant global type. Then \(T\) admits 
  weak elimination of imaginaries.
\end{proposition}

In fact, it suffices to consider sets \(\dX\) in (powers of) a dominant sort.

While the above criterion holds with invariant types, we will actually show 
the existence of a \emph{definable} type in \(\dX\). This is shown in several 
steps. The following result from~\cite{RidVDF} produces an \(\ACVF\)-type 
consistent with a given definable set in a suitable expansion, definable in 
that expansion:

\begin{proposition}[{\cite[Proposition\,9.6]{RidVDF}}]\label{prop:dens def}
  Let \(\hT\supseteq\ACVF^{\G}\) be a theory in a countable language \(\hL\), 
  such that
  \begin{enumerate}
    \item \(\hT\) eliminates imaginaries;
    \item \(\k\) and \(\vg\) are stably embedded, and
    \item The induced theories on \(\k\) and \(\eq{\vg}\) eliminate 
      \(\exists^\infty\).
  \end{enumerate}
  Let \(\hA=\acl_{\hT}(\hA)\subseteq\hN\models\hT\) and let \(\dX\) be a non-empty 
  strict pro-\(\hA\)-definable set of \(\VF\) elements.  Then there exists a 
  global \(\ACVF^\G\)-type \(p\) consistent with \(\dX\), which is 
  \(\hA\)-definable (in \(\hT\)).
\end{proposition}

To replace definability in \(\hT\) with definability in \(\ACVF^\G\), we will 
use the following result. We recall that a subset \(A\) of a structure 
\(\mM\) is \Def{uniformly stably embedded} if for every formula \(\phi(x,y)\) 
there is a formula \(\psi(x,z)\) such that for every \(m\in\mM\) there is 
\(a\in{}A\) such that \(\phi(A,m)=\psi(A,a)\). We have:

\begin{proposition}[{\cite[Corollary\,1.7]{RidSimNIP}}]\label{prop:def 
  enrich}
  Let $T$ be an $\NIP$ $\L$-theory that eliminates imaginaries.  Let 
  $\hT\supseteq T$ be a complete $\hL$-theory that also eliminates 
  imaginaries. Suppose that there exists $\widehat{M}\models\hT$ such that 
  $\restr{\widehat{M}}{\L}$ is uniformly stably embedded in every elementary extension.  
  Let $\hN\models\hT$, $\hA=\dcl_\hT(\hA)\subseteq\hN$ and $p$ be a global 
  \(\L\)-type.  If $p$ is $\hA$-definable in \(\hT\), then it is in fact 
  $\restr{\hA}{\L}$-definable in \(T\).
\end{proposition}

We now go back to our context. We set \(T=\ACVF^\G_{p,p}\) and \(\hT=\eq{(\Tppv)}\). Combining the last two results, 
we obtain:

\begin{corollary}\label{cor:denseDef}
  Let \(\dX\) be a non-empty strict pro-definable set of \(\VF\) elements in 
  \(\Tppv\) (over parameters), and let \(A=\bG(\acl_{\eq{(\Tppv)}}(\code{\dX}))\). Then there is 
  a global $\ACVF^\G_{p,p}$-type \(p\), \(A\)-definable in $\ACVF^\G_{p,p}$ and consistent with 
  \(\dX\).
\end{corollary}
\begin{proof}
  Let \(\hA\) be the algebraic closure of the code of \(\dX\) in 
  \(\hT=\eq{(\Tppv)}\). Corollary\,\ref{cor:st emb pair} shows that the 
  hypothesis of Proposition\,\ref{prop:dens def} holds, and therefore 
  provides us with an \(\hA\)-definable (in \(\hT\)) global \(T\)-type \(p\), 
  consistent with \(\dX\). Now, \cite[Corollary\,A.7]{RidVDF} asserts that 
  the model provided by Proposition\,\ref{prop:complete model pair} satisfies 
  the condition in Proposition\,\ref{prop:def enrich}, and 
  Corollary\,\ref{cor:complete th pair} shows that \(\Tppv\) (hence \(\hT\)) 
  is complete, so  Proposition\,\ref{prop:def enrich} applies to show that 
  \(p\) is definable in \(T\) over \(\bG(\hA) = A\).
\end{proof}

Recall that if \(a\) is an element of a model of \(\SCVH_{p,e}\), we denote by 
\(D_\omega(a)\) the infinite tuple obtained by applying the derivations to 
\(a\).  If \(p(x)\) is a (partial) type in the field sort of \(\SCVH_{p,e}\), we 
let \(\nabla(p)\) be the pro-definable set in the language of \(\ACVF\) 
determined by the condition: \(D_\omega(a)\models\nabla(p)\) if and only if 
\(a\models{}p\) for all tuples \(a\) in a model of \(\SCVH_{p,e}\) (in other words, 
\(\nabla(p)\) is the prolongation of \(p\)).

If \(\dX\) is a definable set in \(\SCVH_{p,e}\) (in some $\VF^n$), then its image \(D_\omega(\dX)\) 
is a strict pro-definable set over the same parameters, and for all types 
\(p\) of \(\SCVH_{p,e}\), \(p\) is consistent with \(\dX\) if and only if 
\(\nabla(p)\) is consistent with \(D_\omega(\dX)\).  Furthermore, any 
complete \(\ACVF_{p,p}\) type consistent with \(D_\omega(\dX)\) is of the form 
\(\nabla(p)\) for a complete \(\SCVF_{p,e}\) type in \(\dX\), over the same 
parameters. Thus, \(\nabla\) provides a bijection between complete 
\(\SCVH_{p,e}\)-types consistent with \(\dX\) and complete \(\ACVF_{p,p}\)-types 
consistent with \(D_\omega(\dX)\), over the same parameters.

\begin{theorem}\label{thm:SCVH dense def}
  Let $\tN\models\SCVH_{p,e}^{\G}$ and $\dX\subseteq\VF^n$ be 
  $\tN$-definable. Let $A=\acl(\code{\dX})$. Then, there exists an 
  $A$-definable type $p$ such that $p(x)\vdash x\in\dX$.
\end{theorem}
\begin{proof}
  Let us assume that $\tN$ is sufficiently saturated and let $N$ denote its 
  algebraic closure. Then $N_P:=(N,\tN)\models\Tppv$. Replacing \(\dX\) with 
  \(D_\omega(\dX)\) as above, it suffices to find an \(A\)-definable \(\ACVF_{p,p}\)-type 
  \(p\).

  Let $B=\eq{\acl}_{\Lppv}(\code{\dX})$.  According to 
  Corollary~\ref{cor:denseDef}, \(\dX\) contains an \(\ACVF_{p,p}^\G\)-type \(p\) 
  definable over \(\bG(B)\). To complete the proof, it remains to show that $\bG(B)$ is contained in $\dcl_N(A)$. 
 To establish this, by Lemma~\ref{L:SCVF-ACVF}(2) it is enough to show that $\bG(B)\cap\tN\subseteq A$. This latter inclusion follows from the fact that $\tN$ is stably embedded in $N_P$ as a pure model of $\SCVH_{p,e}^{\G}$ (by Theorem~\ref{thm:EQ valued}).  
 \end{proof}

Let again \(T\) be an arbitrary theory. In case \(T\) eliminates imaginaries, 
the condition in Prop.~\ref{prop:EICrit} can be viewed as asserting the 
density of invariant types in the space of types over \(\acl(\code{X})\).  
Applying compactness, one obtains:

\begin{proposition}[{\cite[Proposition\,10.2]{RidVDF}}]\label{prop:invTpDense}
  For a theory \(T=\eq{T}\) and a set of parameters \(A\), the following are 
  equivalent:
  \begin{enumerate}
    \item Every \(A\)-definable set contains an \(A\)-invariant type.
    \item \emph{(the invariant extension property over \(A\))} Every type 
      over \(A\) extends to a global \(A\)-invariant type.
  \end{enumerate}
\end{proposition}

Combining everything, we obtain:

\begin{corollary}\label{C:SCVH-EI}
  The theory $\SCVH_{p,e}^{\G}$ eliminates imaginaries and has the invariant 
  extension property, i.e., every type over an algebraically closed set of 
  parameters has a global invariant extension.
\end{corollary}

\begin{proof}
  Weak elimination of imaginaries follows from Theorem~\ref{thm:SCVH dense 
  def} using Proposition \ref{prop:EICrit}. Finite sets are coded by 
  Lemma\,\ref{L:SCVF-ACVF}. Elimination of imaginaries follows.

  The invariant extension property follows again from the same theorem,
  using Proposition \ref{prop:invTpDense}.
\end{proof}

\begin{remark}
When a theory \(T\) is \(\NIP\) and eliminates imaginaries, as is the case 
for $\SCVH_{p,e}^{\G}$, the invariant extension property has various model 
theoretic consequences (cf. \cite[Proposition~2.11]{HruPil}):
\begin{itemize}
  \item Lascar strong type, Kim-Pillay strong type and strong type coincide.  
    That is, over an algebraically closed set \(A\), two points have the same 
    type if and only if there exists a model containing \(A\) over which they 
    have the same type.
  \item Every algebraically closed set is an extension base and thus, by 
    \cite{CheKap}, forking coincides with dividing in \(T\).
\end{itemize}

In fact, since non-forking types are Lascar invariant in \(\NIP\) theories, the invariant extension property is equivalent to the conjunction of the above two conditions.
\end{remark}
\section{Algebraic and definable closure}\label{S:dcl-acl}

In this section, we wish to describe the algebraic and definable closure in 
$\SCVF_{p,e}$ and $\SCVH_{p,e}$. Our main result is that they are no larger than 
what could be expected: they are the (relative) algebraic and definable 
closure in $\ACVF$ of the structure generated by the parameters. We will 
denote by $\langle A\rangle_\lambda$ (respectively $\langle A\rangle_D$) the 
$\LGpe^\lambda$-structure (respectively $\LGpe^D$-structure) generated by 
$A$.

\begin{lemma}\label{lem:acl SCVF field}
  Let $M\models\SCVF_{p,e}$ and $A\substr \VF(M)$ (in $\LDIVpe^{\lambda}$).  
  Then
  \[\VF(\acl(A)) \subseteq \alg{A}.\]
\end{lemma}

\begin{proof}
  Let us first assume that $\val(A)\neq 0$. Since $A$ is closed under 
  $\lambda$-functions, the extension $A\subseteq \VF(M)$ is separable and, as $\VF(M)$ 
  is separably closed, $\sep{A}\subseteq M$. As $A$ contains a $p$-basis of 
  $\VF(M)$, $\sep{A}$ and $\VF(M)$ have the same imperfection degree and 
  hence $\sep{A}\models\SCVF_{p,e}$. By model completeness, 
  $\VF(\acl(A))\subseteq\sep{A}\subseteq\alg{A}$.

  Now, if $\val(A) = 0$, assume that $M$ is sufficiently saturated and let 
  $c\in\ppowi{\VF}(M)$ be transcendental over $\VF(\acl(A))$ and have 
  positive valuation. By the previous paragraph, $\VF(\acl(A)) \subseteq 
  \alg{\VF(\langle A c\rangle_\lambda)} = \alg{A(c)}$. Let $a\in\VF(\acl(A)) 
  \subseteq \alg{A(c)}$. By construction, $c\not\in\alg{A(a)}$ and hence 
  $a\in\alg{A}$.
\end{proof}

A similar argument allows us to reduce the study of $\acl_{\SCVH_{p,e}}$ on 
the field sort to the above result.

\begin{lemma}\label{lem:acl SCVH field}
  Let $M\models\SCVH_{p,e}$ and $A\substr \VF(M)$ (in $\LDIVpe^{D}$). Then
  \[\VF(\acl(A)) \subseteq \alg{A}.\]
\end{lemma}

\begin{proof}
  Let $b$ be a canonical $p$-basis of $\VF(M)$ with $\trdeg(b/\VF(\acl(A))) = 
  e$ (for example, $b$ is a very canonical $p$-basis over $\VF(\acl(A))$). As the 
  $D_{i,n}(x)$ can be expressed as polynomials in $\lambda(x)$ and $b$, it 
  follows that $\VF(\acl(Ab)) \subseteq \VF(\acl_{\SCVF_{p,e}}(A)) \subseteq 
  \alg{\langle A\rangle_\lambda}$. The last inclusion is proved in 
  Lemma\,\ref{lem:acl SCVF field}. Moreover, by \cite[Lemma\,4.3]{ZieSCH}, 
  $\lambda(x)^p$ can be expressed as a polynomial in $D(x)$ and $b$ and hence 
  $\langle A\rangle_\lambda \subseteq \alg{\langle A b\rangle_D} = 
  \alg{A(b)}$. Let $a\in\VF(\acl(A)) \subseteq \alg{A(b)}$.  By construction 
  the tuple $b$ is algebraically independent from $a$ over $A$ and hence 
  $a\in\alg{A}$.
\end{proof}

\begin{lemma}\label{lem:acl field}
  Let $T = \SCVF_{p,e}^\G$ or $T = \SCVH_{p,e}^\G$, $M\models T$ and 
  $A\substr M$. Then
  \[\VF(\acl(A)) = \VF(\acl(\VF(A))) = \alg{\VF(A)}\cap M.\]
\end{lemma}

\begin{proof}
  The second equality follows from Lemmas\,\ref{lem:acl SCVF field} and 
  \ref{lem:acl SCVH field}.

  To prove the first equality it suffices to show that all definable 
  functions from some $\latt_n$ or $\tor_n$ into $\VF$ have a finite image.  
  It is enough to prove this for $T=\SCVF^{\G}_{p,e}$. Consider 
  $M\models\SCVF^{\G}_{p,e}$ of cardinality continuum that contains a dense 
  countable subfield (for example $\sep{\F_p(t_i\mid 0 < i < e)((t_0))}$). In 
  such a model, the sorts $\latt_n$ and $\tor_n$ are countable but any 
  definable subset $\dX$ of (some Cartesian power of) $\VF(M)$ is either 
  finite or has cardinality continuum. To prove that last statement, taking a 
  $\lambda$-resolution we may assume that $\dX$ is quantifier-free $\LDIV(M)$-definable. If 
  it is not finite then, by Corollary\,\ref{cor:red to open in affine space}, 
  there is a a semialgebraic subset $\dX'\subseteq\dX$ which is in 
  $K$-definable bijection with $\bO^d(K)$ for some $d>0$. It follows that 
  $\dX'$ and thus $\dX$ has cardinality continuum.

  As functions with countable domain cannot have an image with cardinality 
  continuum, it follows that any definable function from some $\latt_n$ or 
  $\tor_n$ into $\VF$ has a finite image and hence $\VF(\acl(A)) \subseteq 
  \acl(\VF(A))$.
\end{proof}

\begin{proposition}\label{prop:descr acl}
  Let $T = \SCVF_{p,e}^\G$ or $T = \SCVH_{p,e}^\G$, $M\models T$ and 
  $A\substr M$.
  \[\acl_T(A) = \acl_{\ACVF^{\G}_{p,p}}(A)\cap M\]
\end{proposition}

\begin{proof}
  Pick $a\in \acl_T(A)$. If $a\in\VF$, by Lemma\,\ref{lem:acl field}, 
  $a\in\alg{\VF(A)} \subseteq \acl_{\ACVF_{p,p}^{\G}}(A)$.

  Let us now assume that $a\in \latt_n$ (the same proof also works if $a\in 
  \tor_n$). By quantifier elimination in the geometric language, there is a 
  quantifier free $\LG(A)$-formula $\phi(x)$ such that $\phi(M)$ is finite 
  and $M\models\phi(a)$.  As $\latt_n(\alg{M}) = \latt_n(M)$ and $\phi(x)$ is 
  quantifier free, we have $a\in\phi(M)=\phi(\alg{M})$, with $\phi(\alg{M})$ 
  finite.  It follows that $a\in\acl_{\ACVF_{p,p}^{\G}}(A)$.
\end{proof}

\begin{proposition}\label{prop:descr dcl}
  Let $T = \SCVF_{p,e}^\G$ or $T = \SCVH_{p,e}^\G$, $M\models T$ and 
  $A\substr M$.
  \[\dcl_T(A) = \dcl_{\ACVF_{p,p}^{\G}}(A)\cap M\]
\end{proposition}

\begin{proof}
  Pick $a\in \dcl_T(A)$. If $a\in \latt_n$ or $\tor_n$, then, as above, there 
  is a quantifier free $\LG(A)$-formula $\phi(x)$ such that $\phi(M) = \{a\} 
  = \phi(\alg{M})$ and thus $a\in\dcl_{\ACVF_{p,p}^{\G}}(A)$.

  If $a\in\VF$, by Lemma\,\ref{lem:acl field}, $a\in\alg{\VF(A)}\cap M=:F$. Let 
  $\sigma\in\Aut_{\ACVF_{p,p}^{\G}}(\alg{M}/A)$. The Hasse derivations (and hence 
  the $\lambda$-functions) extend uniquely from $\VF(A)$ to $F$. It 
  follows that $\restr{\sigma}{A\cup F}$ respects all the new 
  structure on $\VF$ in $T$ and therefore $\sigma(a)$ is a $T$-conjugate of 
  $a$ over $A$. In particular $\sigma(a) = a$ and $a\in\dcl_{\ACVF^{\G}_{p,p}}(A)$.
\end{proof}

\section{Metastability}\label{S:metastability}

In this section, we will show that $\SCVF_{p,e}$ is metastable (as defined by 
Haskell, Hrushovski and Macpherson in \cite{HaHrMa08}). But first, let us 
give some definitions.

An $\L(A)$-definable set $\dX$ is said to be \emph{stably embedded} if every 
$\L(M)$-definable set $\dY\subseteq \dX^n$ is $\L(A\cup \dX(M))$-definable. The 
set $\dX$ is said to be \emph{stable} (if it is stably embedded and) if the 
structure on $\dX$ induced by $\L(A)$ is stable. For a thorough discussion of 
(stably embedded) stable sets, we refer the reader to the appendix of 
\cite{ChHr99}. We denote by $\St_A$ the structure whose sorts are the stable 
(stably embedded) sets which are $\L(A)$-definable, equipped with their 
$\L(A)$-induced structure.  We will denote by $\ind_{A}$ forking independence 
in $\St_A$.

\begin{lemma}\label{lem:VF-unstable}
  Let $T=\SCVH_{p,e}$ or $T=\SCVF_{p,e}$, and let $\dX$ be an infinite 
  definable subset of $\VF^n$ for some $n\in\N$.  Then there is a definable 
  function $f:\dX\rightarrow\vg$ with infinite image. In particular, $\dX$ is 
  unstable.
\end{lemma}

\begin{proof}
  We may work over parameters, and it is thus enough to prove the result for 
  $T=\SCVF_{p,e}$. Assume $\dX$ is defined over $K\models\SCVF_{p,e}$.  Using 
  $\lambda$-resolutions, we may assume that $\dX$ is a semialgebraic subset 
  of $K^n$, i.e. defined by a quantifier-free $\LDIV(K)$-formula. By 
  Corollary \ref{cor:red to open in affine space}, there is a a semialgebraic 
  subset $\dX'\subseteq\dX$ which is in $K$-definable bijection with 
  $\O_K^d$, where $d>0$ is the dimension of the Zariski closure of $\dX$. The 
  result follows, by considering the function $f:\O_K^d\rightarrow\vg$ 
  sending $x$ to $\val(x_1)$.
\end{proof}

\begin{proposition}\label{P:StA}
  Let $T=\SCVH^{\G}_{p,e}$ or $T=\SCVF^{\G}_{p,e}$, and let $A\substr 
  M\models T$. Suppose that $\dX$ is an $A$-definable set.  Then the 
  following are equivalent:

  \begin{enumerate}
    \item $\dX$ is stable stably embedded.
    \item $\dX$, expanded by relations for $A$-definable subsets of $\dX^n$ 
      for all $n$, has finite Morley rank.
    \item $\dX\perp\vg$, i.e., any definable subset of $\dX^n\times\vg^m$ is a 
      finite union of rectangles.
    \item There is no definable function $f:\dX\rightarrow\vg$ with infinite 
      image.
    \item $\dX$ is $\rf$-internal.
    \item $\dX$ is $\rf$-analyzable.
    \item Possibly after a permutation of coordinates, $\dX$ is contained in 
      a finite union of sets of the form $s_1/\boldmax s_1\times\cdots\times s_m/\boldmax 
      s_m\times F$, where the $s_i$ are $\acl_T(A)$-definable lattices and 
      $F$ is an $A$-definable finite set of tuples from $\bG(M)$.
  \end{enumerate}
\end{proposition}

\begin{proof}
  The same characterization of stable stably embedded definable sets holds in 
  $\ACVF$ by \cite[Lemma 2.6.2]{HaHrMa06}. Note that (3) and (4) are 
  equivalent, since $\vg$ is stably embedded in $T$ and eliminates 
  imaginaries. The characterization thus holds in $T$ for any $A$-definable 
  set $\dX$ which lives in $\bGim$, by Corollary \ref{cor:purely-geometric}.

  Now let $\dX$ be a definable subset of $\VF^n\times G$, where $G$ is a 
  finite product of sorts from $\bGim$. If the projection of $\dX$ to $\VF^n$ 
  is infinite, the negation of (4) holds by Lemma \ref{lem:VF-unstable}, and 
  the negation of all other statements follows easily from this. We may thus 
  assume that the projection $F$ of $\dX$ to $\VF^n$ is finite, and so we are 
  reduced to definable subsets of $\Gim$.
\end{proof}

\begin{definition}[Stable domination]
  Let $M$ be some $\L$-structure, $C\subseteq M$, $f$ a pro-definable map to 
  $\St_C$ (defined with parameters in $C$) and $a\in M$. We say that 
  $\tp(a/C)$ is \emph{stably dominated via $f$} if for every $a\models p$ and 
  $B\subseteq M$ such that $\St_C(\dcl(CB))\ind_{C} f(a)$,
  \[\tp(B/Cf(a))\vdash\tp(B/Ca).\]

  We say that $p$ is \emph{stably dominated} if it is stably dominated via some map 
  $f$.
\end{definition}

Note that if $p$ is stably dominated, it is stably dominated via any map 
enumerating $\St_C(\dcl(Ca))$ for any $a\models p$.

\begin{remark}\label{rem:suff large}
  In the definition of stable domination, it suffices to show that for any 
  $B$ there exists $B'$ such that $B\subseteq\dcl(B')$ and if 
  $\St_C(\dcl(CB'))\ind_{C} f(a)$, then $\tp(B'/Cf(a))\vdash\tp(B'/Ca)$.

  Indeed, as $\St_C$ is stably dominated, if $\St_C(\dcl(CB))\ind_{C} 
  f(a)$, we may also assume (taking a conjugate of $B'$ over $B$) that 
  $\St_C(\dcl(CB'))\ind_{C} f(a)$ and if $\tp(B'/Cf(a))\vdash\tp(B'/Ca)$, 
  then $\tp(B/Cf(a))\vdash\tp(B/Ca)$.
\end{remark}

\begin{definition}[Metastability]
  Let $T$ be a theory and $\vg$ an $\emptyset$-definable stably embedded set.  
  We say that $T$ is \emph{metastable over $\vg$} if:
  \begin{enumerate}
    \item The theory $T$ has the invariant extension property (as in 
      Corollary \ref{C:SCVH-EI}).
    \item For $M\models T$ sufficiently saturated and for every small subset 
      $A\subseteq M$, there exists a small subset $C\subseteq M$ containing 
      $A$ such that for all tuples $a\in M$, $\tp(a/C\vg(\dcl(Ca)))$ is 
      stably dominated.

      Such a \(C\) is called a metastability basis.
  \end{enumerate}
\end{definition}

Let $T$ be a theory and $U\models T$ a monster model of $T$.
Let $p(x),q(y)\in S(U)$ be definable types. Then one may define the tensor 
product $(p\otimes q)(x,y)\in S(U)$ as follows.  Let $C\subseteq U$ be small 
such that $p$ and $q$ are $C$-definable. Then $p\otimes q$ is the unique 
$C$-definable type in $S(U)$ satisfying
\[(a,b)\models(p\otimes q)| B\text{ if and only if } a\models p|Bb\text{ and 
}b\models q| B\]
for all small \(B\subseteq U\) containing \(C\) and all \((a,b)\in U\). We 
refer to \cite{Sim15} for the basic properties of definable types, 
\(\otimes\) and generically stable types which we will define now.
%Note that \(p \otimes q\) does not depend on the choice of \(C\). It is easy 
%to see that \(\otimes\) is associative.  But it is not always commutative. 

\begin{definition}
  Let \(T\) be \(\NIP\) and \(U\models T\) a monster model.  An invariant 
  type \(p(x)\in S(U)\) is called \emph{generically stable }if \(p(x)\otimes 
  p(y)=p(y)\otimes p(x)\).
\end{definition}

Let us also define orthogonality.

\begin{definition}
  Let \(Q\) be an \(\emptyset\)-definable stably embedded set.  A type \(p\in 
  S(C)\) is said to be \emph{almost orthogonal} to \(Q\) if \(\dcl(Ca)\cap 
  \eq{Q}=\dcl(C)\cap\eq{Q}\) for any \(a\models p\). An invariant type \(p\in 
  S(U)\) is \emph{orthogonal} to \(Q\), denoted by \(p\perp Q\), if \(p| B\) 
  is almost orthogonal to \(Q\) for every set \(B\subseteq U\) over which 
  \(p\) is invariant.
\end{definition}

As a consequence of metastability (and $\NIP$), we get an alternative 
characterization of stably dominated types in case \(\vg\) is totally 
ordered.

\begin{proposition}\label{P:char-st-dom}
  Let \(T\) be an \(\NIP\) theory which is metastable over the stably 
  embedded \(\emptyset\)-definable set \(\vg\).  Assume that \(\vg\) admits a 
  definable total ordering and eliminates imaginaries. A global invariant 
  type $p\in S(U)$ is stably dominated if and only if $p$ is generically 
  stable if and only if  $p\perp\vg$.
\end{proposition}

\begin{proof}
  In \cite[Proposition 2.9.1.a]{HrLo16}, Hrushovski and Loeser show that the 
  above equivalences hold in \(\ACVF\). The proof given there generalizes 
  easily to this more abstract setting.
\end{proof}

In \cite{HaHrMa08}, Haskell, Hrushovski and Macpherson showed that maximally 
complete fields are metastability bases in $\ACVF$. In $\SCVF_{p,e}$ and 
$\SCVH_{p,e}$, we will prove that the situation is quite similar: separably 
maximally complete fields are metastability bases.

But first, let us characterize stably dominated types in $\SCVH_{p,e}$. As 
$\SCVF_{p,e}$ is an enrichment of $\SCVH_{p,e}$ by constants, the same 
results will follow for $\SCVF_{p,e}$. We extend \(D_\omega\) to all of \(\G\) by setting \(D_\omega(a) = a\) for non field points.

\begin{proposition}\label{P:char st dom}
  Let $M\models\SCVH_{p,e}^\G$, $C \subseteq M$ be closed under $D$, 
  $a\in M$ a tuple and $f$ a pro-definable function.  The following are 
  equivalent:
  \begin{enumerate}[(i)]
    \item The type $\tp_{M}(a/C)$ is stably dominated via $f$ (in $M$).
    \item There exists, in $\alg{M}$, a pro-definable function $g$ such that 
      $f = g \circ D_\omega$ and the type $\tp_{\alg{M}}(D_\omega(a)/C)$ is 
      stably dominated via $g$ (in $\alg{M}$).
  \end{enumerate}
\end{proposition}

\begin{proof}
  First, note that the existence of $g$ follows immediately from 
  Proposition\,\ref{prop:descr dcl}.

  By Remark\,\ref{rem:suff large}, to prove stable domination of 
  $\tp_{M}(a/C)$ and $\tp_{\alg{M}}(D_\omega(a)/C)$, it suffices to consider 
  $B = \VF(\dcl_{M}(CB))\subseteq \VF(M)$.

  Moreover, for such a $B$, we have that $\St_C^{M}(B)\ind^{M}_{C} 
  g(D_\omega(a))$ if and only if 
  $\St_C^{\alg{M}}(\dcl_{\alg{M}}(B))\ind^{\alg{M}}_{C} g(D_\omega(a))$.  
  Indeed, by Proposition\,\ref{P:StA}, $\St_C^{M}$ and $\St_C^{\alg{M}}$ are 
  essentially the same structure up to the fact that that $\St_C^{\alg{M}}$ 
  has some more finite sorts that are irrelevant to forking. Also, by quantifier elimination,
  $\tp_{M}(B/Ca)$ is equivalent to $\tp_{\alg{M}}(B/CD_\omega(a))$ and 
  $\tp_{M}(B/Cf(a))$ is equivalent to $\tp_{\alg{M}}(B/Cg(D_\omega(a)))$ 
  (note that we are implicitly using the fact that $Cg(D_\omega(a))$ is 
  closed under $D$ as the image of $g$ is in $\St_C$).

  The equivalence of (i) and (ii) is an immediate consequence.
\end{proof}

\begin{proposition}\label{P:SCVF-met-basis}
  Let $M\models\SCVH_{p,e}$, $C \substr \VF(M)$ be separably maximally 
  complete and $a\in M$ be a tuple, then $\tp(a/C\vg(\dcl(Ca)))$ is 
  stably dominated.
\end{proposition}

\begin{proof}
  Let
  \begin{gather*}
    E = C\vg(\dcl_{\alg{M}}(CD_\omega(a))) = C\vg(\dcl_{M}(Ca))\\
    \alg{E} = 
    %\acl_{\alg{M}}(E) =
     E \cup\perf{C}\\
    g(D_\omega(a)) = \St_E^{M}(\dcl_{M}(ED_\omega(a)))\\
    \alg{g}(D_\omega(a))=\St_\alg{E}^\alg{M}(\dcl_\alg{M}(\alg{E}D_\omega(a)))=
      g(D_\omega(a))\cup\perf{C}
  \end{gather*}
  By Lemma\,\ref{lem:max complete SCVF}, $\perf{C}$ is maximally complete.  
  Thus, 
  $\tp_{\alg{M}}(D_\omega(a)/\alg{E})$ is stably dominated via $\alg{g}$ by \cite[Theorem\,12.18.(ii)]{HaHrMa08}, 
 and  hence $\tp_{\alg{M}}(D_\omega(a)/E)$ is stably dominated via $g$. By 
  Proposition\,\ref{P:char st dom}, $\tp_{M}(a/E) = \tp(a/C\vg(\dcl(Ca)))$ is 
  also stably dominated.
\end{proof}

\begin{corollary}\label{C:SCVH-metastable}
  The theories $\SCVH_{p,e}$ and $\SCVF_{p,e}$ are metastable over $\vg$.
\end{corollary}

%%% Local Variables:
%%% mode: latex
%%% TeX-master: "SCVF-EI"
%%% End:
\section{The stable completion of a definable set in 
\texorpdfstring{$\SCVF$}{SCVF}}

The goal of this section is to generalize a result of Hrushovski and Loeser \cite{HrLo16} on the strict pro-definability of the space of stably dominated types. We show that their proof holds in a context general enough to also encompass separably closed valued fields of finite imperfection degree and Scanlon's theory of contractive valued differential fields (see \cite{Sca00}).

For a proof of the following result, see, e.g., \cite[Remark 2.32]{Sim15}.
\begin{fact}\label{F:Uniform-def}
  Let $T$ be \(\NIP\) and $U\models T$. Then generically stable types are 
  uniformly definable in $T$: for any formula $\phi(x;y)$ there is a formula 
  $\theta(y;z)$ such that for every generically stable type $p(x)\in S(U)$ 
  there is $b\in U$ such that $d_px\phi(x;y)=\theta(y,b)$.
\end{fact}

Hrushovski and Loeser \cite{HrLo16} use this fact, together with Proposition 
\ref{P:char-st-dom}, to encode the set $\widehat{\dX}$ of global stably 
dominated types concentrating on some definable set $\dX$ in \(\ACVF\) as a 
pro-definable set.

\begin{notation*}
  Given an $\emptyset$-definable stably embedded set $\Sort{Q}$, a $C$-definable set 
  $\dX$ and a set $A$ containing $C$, $S_{\mathrm{def},\dX}(A)$ denotes the set 
  of global $A$-definable types $p(x)$ such that $p(x)\vdash x\in \dX$, and 
  $S_{\mathrm{def},\dX}^{\Sort{Q}}(A):=\{p\in S_{\mathrm{def},\dX}(A)\, |\, p\perp \Sort{Q}\}$.  
  Finally, $\widehat{\dX}(A):=\{p\in S_{\mathrm{def},\dX}(A)\, |\, p\text{ is 
  stably dominated}\}$.
\end{notation*}

\begin{fact}[{\cite[Lemma 2.5.1]{HrLo16}}]\label{F:HL-prodef}
  Assume $T$ eliminates imaginaries. Let $\Sort{Q}$ be an $\emptyset$-definable 
  stably embedded set. Assume that the definable types orthogonal to $\Sort{Q}$ are 
  uniformly definable.  Then for any $C$-definable set $\dX$, 
  $S_{\mathrm{def},\dX}^{\Sort{Q}}$ is $C$-pro-definable: there is a $C$-pro-definable set $\Sort{Z}$ such that for any $A\supseteq C$ there is a 
  canonical bijection $\Sort{Z}(A) \simeq S_{\mathrm{def},\dX}^{\Sort{Q}}(A)$.

  Moreover, if $f:\dX\rightarrow \Sort{Y}$ is a definable function, then, identifying 
  $S_{\mathrm{def},\dX}^{\Sort{Q}}$ and $S_{\mathrm{def},\Sort{Y}}^{\Sort{Q}}$ with the corresponding 
  pro-definable sets, the map 
  $f_{\mathrm{def},\dX}:S_{\mathrm{def},\dX}^{\Sort{Q}}\rightarrow S_{\mathrm{def},\Sort{Y}}^{\Sort{Q}}$, 
  $p\mapsto f_\star(p)$ is pro-definable.
\end{fact}

%In \cite{HrLo16}, the authors assume that $Q$ eliminates imaginaries, and 
%they do not state the pro-definability of 
%$f_{\mathrm{def},X}:S_{\mathrm{def},X}^Q\rightarrow S_{\mathrm{def},Y}^Q$.  
%The statement we give follows from the proof of \cite[Lemma 2.5.1]{HrLo16}.
We briefly sketch the argument. For notational simplicity, we will assume 
$C=\emptyset$.
Let $f:\dX\rightarrow\eq{\Sort{Q}}$ be definable (with parameters) and let $p\in 
S_{\mathrm{def},\dX}^{\Sort{Q}}(U)$.  As $p\perp \Sort{Q}$, $f_\star(p)$ is a realized type, 
i.e., there is $\gamma\in \eq{\Sort{Q}}$ such that $\mathrm{d}_{p} x(f(x)=\gamma)$.  
We will denote this by $p_\star(f)=\gamma$.

Now let $f:\dX\times \Sort{W}\rightarrow \eq{\Sort{Q}}$ be $\emptyset$-definable, 
$f_w:=f(-,w)$.
It follows from the assumptions that there is a set $\Sort{S}$ and a function 
$g:\Sort{S}\times \Sort{W}\rightarrow \eq{\Sort{Q}}$, both $\emptyset$-definable, such that for 
every $p\in S_{\mathrm{def},\dX}^\Sort{Q}(U)$, the function 
\[p_\star(f):\Sort{W}\rightarrow\eq{\Sort{Q}}, \,\,w\mapsto p_\star(f_w)\]
is equal to $g_s=g(s,-)$ for a unique $s\in \Sort{S}$.

Now choose an enumeration $f_i:\dX\times \Sort{W}_i\rightarrow\eq{\Sort{Q}}$ ($i\in I$) of 
the functions as above (with corresponding $g_i:S_i\times 
\Sort{W}_i\rightarrow\eq{\Sort{Q}}$).  Then \[p\mapsto c(p):=\{(s_i)_{i\in I}\mid 
p_\star(f_i)=g_{i,s_i } \text{ for all $i$}\}\] defines an injection of 
$S_{\mathrm{def},\dX}^{\Sort{Q}}$ into $\prod_{i\in I}\Sort{S}_i$, and one may show that the 
image $\Sort{Y}_i$ of $c(S_{\mathrm{def},\dX}^{\Sort{Q}})$ under the projection map to $\Sort{S}_i$ is 
$\infty$--definable. Since the set of $\Sort{S}_i$'s is closed under taking finite 
products (this may be seen using products of the corresponding $f_i$'s), 
pro-definability of $c(S_{\mathrm{def},\dX}^{\Sort{Q}})$ follows.

\begin{corollary}\label{C:def-types-pro-def}
  Let $T$ be a theory which eliminates imaginaries, and let $\dX$ be a 
  $C$-definable set.
  \begin{enumerate}
    \item Assume $T$ is stable. Then $S_{\mathrm{def},\dX}$ is canonically a 
      $C$-pro-definable set.
    \item Assume $T$ is \(\NIP\) and metastable over the stably embedded 
      $\emptyset$-definable set $\vg$.  Assume that $\vg$ admits a definable 
      total ordering and eliminates imaginaries. Then $\widehat{\dX}$ is 
      canonically $C$-pro-definable. Moreover, if $f:\dX\rightarrow \Sort{Y}$ is a 
      definable function, then the map 
      $\widehat{f}:\widehat{\dX}\rightarrow\widehat{\Sort{Y}}$, $p\mapsto f_\star(p)$ 
      is pro-definable, once $\widehat{\dX}$ and $\widehat{\Sort{Y}}$ are identified 
      with the corresponding pro-definable sets.
  \end{enumerate}
\end{corollary}

\begin{proof}
  Both parts follow from Fact \ref{F:HL-prodef}. For (1), note that if $\Sort{Q}$ is 
  a 2-element set, then $S_{\mathrm{def},\dX}^{\Sort{Q}}=S_{\mathrm{def},\dX}$. Since in a 
  stable theory all definable types are generically stable, uniform 
  definability of types follows from Fact \ref{F:Uniform-def}. 

  In (2), one has $\widehat{\dX}(A)=S_{\mathrm{def},\dX}^{\vg}(A)$ for all 
  $A\supseteq C$, by Proposition\,\ref{P:char-st-dom}.  As in the proof of 
  (1), uniform definability holds by Fact \ref{F:Uniform-def}.
\end{proof}

Recall that a theory $T$ has the \emph{finite cover property} if there is a 
formula $\phi(x,y)$ (where $x$ and $y$ may be tuples of variables) such that 
for every $n\in\N$ there are $a_1,\ldots,a_n\in U\models T$ such that 
$\models\neg\exists x\bigwedge_{i\leq n}\phi(x,a_i)$ and $\models\exists 
x\bigwedge_{i\leq n,i\neq k}\phi(x,a_i)$ for every $k\leq n$.  The theory $T$ 
is \emph{nfcp} if it does not have the finite cover property. By a result of 
Shelah \cite[II Theorems\,4.2, 4.4]{ShClass}, $T$ is nfcp if and only if $T$ is stable and $\eq{T}$ 
eliminates $\exists^\infty x$.

The following characterization is due to Poizat.

\begin{fact}[{\cite[Th\'eor\`eme\,5]{PoiBellesPaires}}]\label{F:Poizat-nfcp}
  Let $T$ be stable. Then $T$ is nfcp if and only if for every pair of 
  formulas $\phi(x;y)$ and $\theta(y;z)$ the set \[\Sort{D}_{\phi,\theta}=\{c\in U\, 
  | \, \theta(y,c)\text{ is the $\phi$-definition of a (complete) global 
  type}\}\] is definable.
\end{fact}

\begin{corollary}\label{C:nfcp-strictpro}
  Let $T$ be stable. Then $T$ is nfcp if and only if for every definable set 
  $\dX$, the set $S_{\mathrm{def},\dX}$ is strict pro-definable.
\end{corollary}

\begin{proof}
  Let $\Sort{Z}=\varprojlim \Sort{Z}_i$ be the pro-definable set given by the proof of Fact 
  \ref{F:HL-prodef}, with $\Sort{Z}(A)\simeq S_{\mathrm{def},\dX}(A)$ canonically.  
  Then $\Sort{Z}_i$ corresponds to canonical parameters of instances of a formula 
  $\theta=\theta_\phi(y;z)$, and $\pi_i(\Sort{Z})\subseteq \Sort{Z}_i$ is precisely the set 
  of those parameters corresponding to $\phi$-definitions of complete global 
  types, i.e., $\pi_i(\Sort{Z})=\Sort{D}_{\phi,\theta}$.  We conclude by Fact 
  \ref{F:Poizat-nfcp}.
\end{proof}

Hrushovski and Loeser showed that for every definable set $\dX$ in \(\ACVF\), 
the set $\widehat{\dX}$ of stably dominated types concentrated on $\dX$ is strict 
pro-definable \cite[Theorem 3.1.1]{HrLo16}.  Analyzing their proof, we obtain 
the following abstract version of this important technical result, which 
covers the theories we are interested in.

\begin{theorem}\label{T:criterion-strictpro}
  Let $T$ be a complete theory which eliminates imaginaries, is \(NIP\) and 
  metastable over some $\emptyset$-definable stably embedded set $\vg$. Let 
  $\rf_{i\in I}$ be $\emptyset$-definable stable stably embedded sets. Assume 
  the following properties hold:
  \begin{enumerate}
    \item $\vg$ eliminates imaginaries and admits a definable total ordering;
    \item every set of parameters $A$ is included in a metastability basis $C$ such that for every parameter set $B = \dcl(CB)$, $\St_C(B)$ 
      is interdefinable with $\bigcup_i \rf_i(B)$;
    \item The (multisorted) structure $\bigcup_i \rf_i$ (with the induced 
      structure) is nfcp.
  \end{enumerate}
  Then, for every $A$-definable set $\dX$, the set $\widehat{\dX}$ is strict 
  $A$-pro-definable.
\end{theorem}

\begin{proof}
  Let $\dX$ be $A$-definable. Pro-definability of $\widehat{\dX}$ is the content 
  of Corollary \ref{C:def-types-pro-def}(2). We recall the construction of 
  the pro-definable set $\Sort{D}$ satisfying $\Sort{D}(B)=\widehat{\dX}(B)$ canonically, for 
  every parameter set $B\supseteq A$.  Let $f_i:\dX\times 
  \Sort{W}_i\rightarrow\vg^{n_i}$ ($i\in I$) be an enumeration of all $A$-definable 
  families of functions from $\dX$ to $\eq{\vg}$.  As $\vg$ eliminates 
  imaginaries, it is enough to consider functions with values in $\vg^{n}$ 
  for some $n$. Let $g_i:\Sort{S}_i\times \Sort{W}_i\rightarrow\vg^{n_i}$ be $A$-definable 
  such that for any $p\in\widehat{\dX}(U)$ the function 
  $p_\star(f_i):\Sort{W}_i\rightarrow\vg^{n_i}$ is equal to $g_{i,s_i}=g_i(s_i,-)$ for 
  some unique $s_i\in \Sort{S}_i$.  Then $\Sort{D}$ is the image of the injective map 
  $c:\widehat{\dX}\hookrightarrow\prod_{i\in I}\Sort{S}_i$, $p\mapsto(s_i)_{i\in I}$. 

  In order to show that $\Sort{D}$ is strict pro-definable, it is enough to show 
  that the projection $\Sort{Y}_i:=\pi_i(\Sort{D})\subseteq \Sort{S}_i$ is definable for all $i\in 
  I$. (Note that the $\Sort{S}_i$'s are closed under taking finite products.) We 
  already know that $\Sort{Y}_i$ is $\infty$-definable, so there only remains to 
  show that it is a union of definable sets, i.e.,
  $\mathrm{ind}$-definable. Now fix $i\in I$. We will omit indices and write 
  $f:\dX\times \Sort{W}\rightarrow\vg^n$ and $g:\Sort{S}\times \Sort{W}\rightarrow \vg^n$ in what 
  follows. 

  Let $\Sort{Z}$ be the set of functions $g_0 : \Sort{W} \to \vg^n$ such that there exist:
  \begin{itemize}
    \item a finite product $\Sort{L} = \prod_j \rf_j^{n_j}$;
    \item a function $h : \dX \to \Sort{L}$, definable with parameters $c$, and
    \item for $\phi(y;c,v,w)= (\exists x\in \dX\,h(x) = y)\wedge (\forall x\in 
      \dX\,h(x) = y \to f(x,w) = v)$, a definable $\phi(y;c,v,w)$-type $q_0$ 
      concentrating on $\Sort{L}$
  \end{itemize}
  satisfying $g_0(w) = \gamma\text{ if and only if 
  }d_{q_0}y\,\phi(y;c,\gamma,w)$.

  \smallskip

  By Fact \ref{F:Poizat-nfcp}, for every formula $\theta(v,w;z)$, the fact 
  that, for a given $t$, $\theta(v,w,t)$ is the $\phi(y;c,v,w)$-definition of 
  a consistent type is a definable condition in $t$. It follows that $\Sort{Z}$ is 
  $\mathrm{ind}$-definable. We will now show that $\Sort{Z} = \Sort{Y}$.

  First, pick $p\in\widehat{\dX}$. Let $C$ be a set as in Hypothesis 2 such 
  that $C \supseteq A$ and $p$ is $C$-definable. By stable domination, 
  Hypothesis 2 (and compactness), there exists $h$ and $\Sort{L}$ as above such that 
  for all $a\models\restr{p}{C}$ and all $b$ and $\gamma$, if $h(a)\ind_C 
  \St_C(\dcl(Cb\gamma))$, $\tp(b,\gamma,h(a)) \vdash f(a,b) = \gamma$.  
  Actually, making $h$ bigger we may assume that the above holds for all 
  $a\in \dX$. Let $q = h_\star p$, then $\restr{q}{Cb\gamma}(y)\vdash \forall 
  x\in \dX\,h(x) = y \to f(x,b) = \gamma$ and the tuple 
  $(\Sort{L},h,\restr{q}{\theta})$ proves that $p_\star f \in \Sort{Z}$ and hence 
  $\Sort{Y}\subseteq \Sort{Z}$.

  Now, let $g_0\in \Sort{Z}$ and $\Sort{L}$, $h$ and $q_0$ as in the definition of $\Sort{Z}$. Let 
  $C$ be as in Hypothesis 2, such that $g_0$, $h$ and $q_0$ are defined over 
  $C$. Let $b\models \restr{q_0}{C}$ and $a\in \dX$ such that $h(a) = b$. Let 
  $C' = \acl(C\vg(\dcl(Ca)))$. Since $C$ is a metastability basis, 
  $\tp(a/C')$ is stably dominated (and thus in particular definable over $C'$), so $h_\star p$ is definable over $C'$.  
  By orthogonality of $\vg$ and the stable part, $h_\star p$ is definable 
  over $\St_C(C')=\St_C(C)$, which is interdefinable with $C$. As $h_\star p$ extends 
  $\restr{q_0}{C}$, we have that $h_\star p = q_0$. Let $b\in \Sort{W}$, 
  $\gamma = g_0(b)$ and $a\models\restr{p}{C'b\gamma}$, then 
  $h(a)\models\restr{q_0}{Cb\gamma}$ and, by definition of $\Sort{Z}$, $f(a,b) = 
  \gamma$. It follows that $p_\star f(b) = g_0(b)$, hence $g_0 = p_\star f 
  \in \Sort{Y}$ and so $\Sort{Z}\subseteq \Sort{Y}$.
\end{proof}

\begin{corollary}\label{C:SCVF-strict-pro}
  Let $T=\SCVH_{p,e}$ (or any completion of $\SCVF_{p,e}$).  Then for every 
  $A$-definable set $\dX$, $\widehat{\dX}$ may be canonically identified with a 
  strict $A$-pro-definable set.
\end{corollary}

\begin{proof}
  This follows from Theorem \ref{T:criterion-strictpro}.  Indeed, 
  $\SCVH_{p,e}$ is \(\NIP\) (Corollary \ref{C:SCVF-NIP}), and it is 
  metastable over $\vg$ by Corollary \ref{C:SCVH-metastable}. Since $\vg$ 
  is a stably embedded pure divisible ordered Abelian group, it eliminates 
  imaginaries. Any parameter set $A$ is contained in a separably maximal 
  model $K$, and such a $K$ is a metastability basis (see Proposition 
  \ref{P:SCVF-met-basis}). 

  Over $K$, indeed over any model of $\SCVH_{p,e}$, $\St_K(B)=\dcl(\rf(B))$ 
  for every parameter set $B=\dcl(BK)$.  This follows from Proposition 
  \ref{P:StA}(7) and the fact that every lattice $s\in \latt_n(K)$ has a 
  $K$-definable basis. Since the residue field $\rf$ is a pure model of 
  ACF$_p$, it is nfcp. (Note that purity and stable embeddedness of $\vg$ 
  and $\rf$ follow from the corresponding results in $\ACVF_{p,p}$, by 
  Corollary \ref{cor:purely-geometric}.) Thus, all hypotheses of Theorem 
  \ref{T:criterion-strictpro} are satisfied.
\end{proof}

We now discuss a similar context in which Theorem 
\ref{T:criterion-strictpro} applies. Let \(\VDF\) be the theory 
of existentially closed valued differential fields $(K,v,\partial)$ of 
residue characteristic 0 satisfying $v(\partial(x))\geq v(x)$ for all $x$.  
This theory had first been studied by Scanlon \cite{Sca00}. The theory \(\VDF\) is \(\NIP\), and the residue field is stably embedded 
and a pure model of DCF$_0$, with the derivation induced by $\partial$. 

The third author has shown in \cite{RidVDF} that \(\VDF\) 
eliminates imaginaries in the geometric sorts and that it is metastable 
over the value group $\vg$, a stably embedded pure divisible ordered 
Abelian group. As shown there, every set of parameters is included in a metastability basis which is a model, and over any model any stably embedded stable set is interdefinable with the residue field, since it is the case in the underlying algebraically closed valued field. As DCF$_0$ is stable and eliminates imaginaries, it is enough to 
show that it eliminates $\exists^\infty x$. But this follows from the fact 
that the algebraic closure of a set $A$ is the field theoretic algebraic 
closure of the differential field generated by $A$, by quantifier 
elimination.

We have thus proved the following result.

\begin{corollary}\label{C:VDF-strict-pro}
  Let $\dX$ be any $A$-definable set in a model of \(\VDF\). Then 
  $\widehat{\dX}$ may be canonically identified with a strict 
  $A$-pro-definable set.
\end{corollary}
\bibliography{SCVF-EI}

\end{document}